\documentclass[preprint,12pt]{elsarticle}

\usepackage{amssymb}
\usepackage{amsthm}

\usepackage{amsmath}

\usepackage[dvipsnames]{xcolor}

\newtheorem{thm}{Theorem}[section]

\newtheorem{lemma}[thm]{Lemma}
\newtheorem{cor}[thm]{Corollary}
\newtheorem{ques}{Question}

\theoremstyle{definition}
\newtheorem{defn}[thm]{Definition}
\newtheorem{example}[thm]{Example}

\theoremstyle{remark}

\newtheorem*{claim}{Claim}

\journal{Annals of Pure and Applied Logic}

\begin{document}

\begin{frontmatter}



\title{Parametrizing the Ramsey theory of vector spaces I: Discrete spaces}


\author[1]{Iian B. Smythe\fnref{fn1}}

\affiliation[1]{organization={Department of Mathematics and Statistics, University of Winnipeg},
            addressline={515 Portage Avenue}, 
            city={Winnipeg}, 
            state={Manitoba},
            postcode={R3B 2E9},
            country={Canada}}

\fntext[fn1]{The author is partially supported by a Natural Sciences and Engineering Research Council of Canada Discovery Grant (RGPIN-2023-03343).}

\begin{abstract}
	We show that the Ramsey theory of block sequences in infinite-dimensional discrete vector spaces can be parametrized by perfect sets. As special cases, we prove combinatorial dichotomies for definable families of partitions and linear transformations on those spaces. We also consider the extent to which analogues of selective ultrafilters in this setting are preserved by Sacks forcing.
\end{abstract}




\end{frontmatter}


\section{Introduction}\label{sec:intro}

The aim of this article is to present parametrized versions of Ramsey-theoretic dichotomies for infinite-dimensional vector spaces. Our results will typically take the form: for any suitably definable family of partitions, parametrized by reals, of these spaces or related structures, there is an infinite-dimensional subspace on which uncountably many of those partitions all have the same prescribed extremal behavior. We will explore consequences of these results, both mathematical and metamathematical.

Our setting is a countably infinite-dimensional discrete vector space $E$ over a countable, possibly finite, field $F$, having a distinguished basis $(e_n)_{n\in\omega}$. For concreteness, one may take $E=\bigoplus_{n\in\omega} F$ and let $e_n$ be the $n$th unit coordinate vector. The set of nonzero vectors in $E$ will be denoted by $E^\times$. For any vector $v\in E^\times$, its \emph{support} is the finite set $\mathrm{supp}(v)$ of indices $i\in\omega$ such that $e_i$ has a nonzero coefficient when $v$ is expressed as a linear combination of the $e_n$'s. By a \emph{subspace} of $E$, we will always mean a linear subspace. Throughout, the zero vector will typically be ignored.

The primary obstacle to a satisfactory infinitary\footnote{There is an analogue of the finite form of Ramsey's Theorem, for finite-dimensional vector spaces over finite fields, due to Graham, Leeb, and Rothschild \cite{MR0306010}.} Ramsey theory for vector spaces is the failure of a natural analogue of the pigeonhole principle. For any finite partition of $E^\times$, an infinite-dimensional subspace is \emph{homogeneous} if all of its nonzero vectors are entirely contained in one piece of the partition. When $|F|=2$, vectors may be identified with their supports in $\mathrm{FIN}$, the set of all nonempty finite subsets of $\omega$, and Hindman's Theorem \cite{MR0349574} says that any finite partition of $\mathrm{FIN}$ admits an infinite-dimensional homogeneous subspace. However, this fails over all other fields due to the existence of asymptotic pairs.

\begin{defn}
	\begin{enumerate}[(a)]
		\item A set $A\subseteq E^\times$ is \emph{asymptotic} if it meets every infinite-dimensional subspace of $E$.
		\item An \emph{asymptotic pair} is a pair of disjoint asymptotic sets.
	\end{enumerate}
\end{defn}

\begin{example}\label{ex:asym_pair}
	Suppose that $|F|>2$. The \emph{oscillation} of a nonzero vector $v$ is the number of times the nonzero coefficients change in its basis expansion with respect to $(e_n)_{n\in\omega}$, read from left-to-right. That is, if $\mathrm{supp}(v)=\{{n_0},\ldots,{n_k}\}$, where $n_0<\cdots<n_k$, and $v=\sum_{i=0}^k a_ie_{n_i}$, then
	\[
		\mathrm{osc}(v)=|\{i<k:a_i\neq a_{i+1}\}|.
	\]	
	It is shown in the proof of Theorem 7 in \cite{MR2737185} that on any infinite-dimensional subspace of $E$, the range of $\mathrm{osc}$ contains arbitrarily long intervals. It follows that the sets 
	\begin{align*}
		A_0 &= \{v\in E^\times:\mathrm{osc}(v)\text{ is even}\}\\
		A_1 &= \{v\in E^\times:\mathrm{osc}(v)\text{ is odd}\}
	\end{align*}
	form an asymptotic pair. Note that the oscillation map, and thus the sets $A_0$ and $A_1$, are invariant under multiplication by nonzero scalars.
\end{example}

This failure of the pigeonhole principle means that we can only assert the following, in general:
\begin{quote}\label{weak_PHP}
	($\dagger$) \emph{For any $A\subseteq E^\times$, either there is an infinite-dimensional subspace $V$ of $E$ such that $A\cap V=\emptyset$ or $A$ is asymptotic.}
\end{quote}
Here, we view $A$ as a partition of $E^\times$ into $A$ and its complement. While the statement ($\dagger$) is trivially true by definition of ``asymptotic'', we believe that it is the appropriate ``weak pigeonhole principle'' for infinite-dimensional vector spaces, something that will become increasingly clear in the parametrized setting (see Section 3 of \cite{MR4516184} for a related discussion).

The primary results we wish to ``parametrize'' are dichotomies for definable partitions of certain spaces of sequences in $E$. A \emph{block sequence} in $E$ is a sequence of nonzero vectors $x_n$ such that
\[
	\max(\mathrm{supp}(x_n))<\min(\mathrm{supp}(x_{n+1}))
\] 
for all $n$ less than the length of the sequence (which may be finite or $\omega$). Note that any block sequence is linearly independent. It is easy to see that every infinite-dimensional subspace of $E$ contains an infinite block sequence (Lemma 2.1 in \cite{MR3864398}).

The set of all infinite block sequences in $E$ will be denoted by $E^{[\infty]}$. As $E$ is discrete, $E^{[\infty]}$ inherits a Polish topology, and thus a structure of Borel and analytic sets, as a subspace of the product space $E^\omega$. The set of all finite block sequences will be denoted by $E^{[<\infty]}$.

We write $\langle X\rangle$ for the linear span of a block sequence $X$, and order block sequences according to their spans: $Y\preceq X$ if $\langle Y\rangle\subseteq\langle X\rangle$, in which case we abuse terminology and say that $Y$ is a \emph{block subsequence} of $X$. For $m\in\omega$, we write $X/m$ (or $X/\vec{v}$ for a finite sequence of vectors $\vec{v}$) for the tail of $X$ whose supports are strictly above $m$ (or that of $\vec{v}$). We write $Y\preceq^* X$ if $Y/m \preceq X$ for some $m\in\omega$. More generally, if $V$ is a subspace of $E$, we write $V/m$ (or $V/\vec{v}$) for $V\cap\langle(e_n)_{n>m}\rangle$ (or $V\cap\langle(e_n)_{n>\max(\mathrm{supp}(\vec{v}))}\rangle$).

The ordering on block sequences has the following important property: if $(X_n)_{n\in\omega}$ is a sequence in $E^{[\infty]}$ such that $X_{n+1}\preceq^* X_n$ of all $n\in\omega$, then there is an $X\in E^{[\infty]}$ such that $X\preceq^* X_n$ for all $n\in\omega$. Such an $X$ is called a \emph{diagonalization} of the sequence $(X_n)_{n\in\omega}$.

The dichotomies for block sequences are phrased in terms of certain games. Given an $X\in E^{[\infty]}$, the \emph{asymptotic game} $F[X]$ played below $X$ is the two player game where player I goes first and plays a natural number $n_0$, and player II responds with a nonzero vector $y_0\in\langle X/n_0\rangle$. They continue in this fashion, alternating with I playing an $n_k\in\omega$ and II playing a $y_k\in\langle X/n_k\rangle$ for each $k\in\omega$. We further demand that II's moves form a block sequence.
\[
	F[X]:\quad
	\begin{matrix}
		\mathrm{I}  & n_0 &       & n_1 &       & n_2 &       &\cdots\\
		\mathrm{II} &	   & y_0 &       & y_1 &       & y_2 &\cdots
	\end{matrix}
\]

The \emph{Gowers game} $G[X]$ played below $X$ is defined similarly, with I now playing infinite block sequences $Y_k\preceq X$ and II responding with nonzero vectors $y_k\in\langle Y_k\rangle$ for each $k\in\omega$, again forming a block sequence.
\[
	G[X]:\quad
	\begin{matrix}
		\mathrm{I}  & Y_0 &       & Y_1 &       & Y_2 &       &\cdots\\
		\mathrm{II} &	   & y_0 &       & y_1 &       & y_2 &\cdots
	\end{matrix}
\]
Note that $F[X]$ may be viewed as a variant of $G[X]$ where I is restricted to playing tails of $X$. In both games, the \emph{outcome} is the block sequence $(y_k)_{k\in\omega}$ consisting of II's moves. 

\emph{Strategies} for the players in these games can be defined as functions from all finite sequences of possible previous moves by the opposing player to valid moves by the current player. Given a set $\mathbb{A}\subseteq E^{[\infty]}$, we say that a player has a strategy for \emph{playing into (or out of)} $\mathbb{A}$ if they have a strategy all of whose outcomes lie in (or out of) $\mathbb{A}$. We stress again that the ``outcomes'' here are the result of II's moves, even when referring to a strategy for I.

Isolating the combinatorial content of a dichotomy for block sequences in Banach spaces, due to Gowers \cite{MR1954235}, Rosendal proved the following:

\begin{thm}[Rosendal \cite{MR2604856}]\label{thm:Rosendal}
	If $\mathbb{A}\subseteq E^{[\infty]}$ is analytic, then there is an $X\in E^{[\infty]}$ such that either:
	\begin{enumerate}[(1)]
		\item I has a strategy in $F[X]$ for playing out of $\mathbb{A}$, or
		\item II has a strategy in $G[X]$ for playing into $\mathbb{A}$.
	\end{enumerate}
\end{thm}

The conclusions (1) and (2) in Theorem \ref{thm:Rosendal} are mutually exclusive and provide a strategy for the player in the game which is \textit{a priori} more difficult for that player. If $\mathbb{A}$ is a clopen set in $E^{[\infty]}$ which depends only on the first coordinate, then $\mathbb{A}$ induces a partition of $E^\times$ and the games $F[X]$ and $G[X]$ are decided after one move; in this case, Theorem \ref{thm:Rosendal} reduces to a form of the weak pigeonhole principle ($\dagger$) above for the subspace spanned by $X$.

In \cite{MR3864398}, the author showed that Theorem \ref{thm:Rosendal} could be ``localized'' in the sense that the witness $X$ to the conclusion could always be found in a prescribed $(p^+)$-family (defined in Section \ref{sec:families} below) of block sequences:

\begin{thm}[Smythe \cite{MR3864398}]\label{thm:local_Rosendal}
	Let $\mathcal{H}$ be a $(p^+)$-family on $E$. If $\mathbb{A}\subseteq E^{[\infty]}$ is analytic, then there is an $X\in\mathcal{H}$ such that either:
	\begin{enumerate}[(1)]
		\item I has a strategy in $F[X]$ for playing out of $\mathbb{A}$, or
		\item II has a strategy in $G[X]$ for playing into $\mathbb{A}$.
	\end{enumerate}
\end{thm}

It is further shown in \cite{MR3864398} that if the family $\mathcal{H}$ satisfies the additional property of being \emph{strategic} (also defined in Section \ref{sec:families}), then Theorem \ref{thm:local_Rosendal} can be extended to all ``reasonable definable'' sets $\mathbb{A}\subseteq E^{[\infty]}$, i.e., those in the inner model $\mathbf{L}(\mathbb{R})$, assuming large cardinal hypotheses. While the set $E^{[\infty]}$ of all block sequences is trivially a strategic $(p^+)$-family, such families may be much smaller; they can (consistently) be filters with respect to the ordering on block sequences. Strategic $(p^+)$-filters can then be characterized as exactly those filters which are generic over $\mathbf{L}(\mathbb{R})$ for the partial order $(E^{[\infty]},\preceq)$, in the sense of forcing, and are said to have ``complete combinatorics''.

This progression of results mirrors those for the space $[\omega]^\omega$ of infinite subsets of the natural numbers $\omega$: theorems of Galvin--Prikry \cite{MR0337630} and Silver \cite{MR0332480} established that if $A$ is an analytic subset of $[\omega]^\omega$, then there is an $x\in[\omega]^\omega$ all of whose infinite subsets are either contained in or disjoint from $A$, a far-reaching generalization of Ramsey's Theorem. These results were then localized to selective coideals by Mathias \cite{MR0491197} and to semiselective coideals by Farah \cite{MR1644345}. Todor\v{c}evi\'{c} showed that selective ultrafilters are generic over $\mathbf{L}(\mathbb{R})$ for $([\omega]^{\omega},\subseteq)$ under large cardinal hypotheses, see again \cite{MR1644345}.

Recall that a subset of a Polish space is \emph{perfect} if it is closed and has no isolated points. Nonempty perfect sets necessarily contain homeomorphic copies of the Cantor space $2^\omega$ and thus are uncountable. The paradigmatic ``parametrized'' Ramsey theorem is the following result, due independently to Miller and Todor\v{c}evi\'{c}:

\begin{thm}[Miller--Todor\v{c}evi\'{c} \cite{MR983001}]\label{thm:para_Silver}
	If $\mathbb{A}\subseteq \mathbb{R}\times[\omega]^\omega$ is analytic, then there is a nonempty perfect set $P\subseteq \mathbb{R}$ and an $x\in[\omega]^\omega$ such that either $P\times [x]^\omega \subseteq \mathbb{A}$ or $(P\times [x]^\omega)\cap \mathbb{A}=\emptyset$.
\end{thm}

In other words, if we are given a suitably definable family $\{\mathbb{A}_t:t\in\mathbb{R}\}$ of partitions of $[\omega]^\omega$, then we can find a perfect subfamily $\{\mathbb{A}_t:t\in P\}$ and a set $x$ which is homogeneous for all of these perfectly many partitions simultaneously and in the same way. By standard facts, $\mathbb{R}$ can be replaced by any uncountable Polish space here. Theorem \ref{thm:para_Silver}  was localized to semiselective coideals by Farah in \cite{MR1644345}. The metamathematical counterpart to Theorem \ref{thm:para_Silver}, due to Baumgartner and Laver, is that selective ultrafilters are preserved by Sacks forcing $\mathbb{S}$:

\begin{thm}[Baumgartner--Laver \cite{MR556894}]\label{thm:Baum_Laver}
	If $\mathcal{U}$ is a selective ultrafilter on $\omega$ and $g$ is $\mathbf{V}$-generic for $\mathbb{S}$, then $\mathcal{U}$ generates a selective ultrafilter in $\mathbf{V}[g]$.
\end{thm}

Similar parametrized Ramsey theorems have been established in other settings, see \cite{MR2285192}, \cite{MR2875907}, \cite{MR2603812}, and \cite{yyz_thesis}. Particularly relevant to the present work are results of Zheng \cite{MR3717938} on parametrized forms of Milliken's theorem \cite{MR0373906} for the space of infinite block sequences in $\mathrm{FIN}$, and Kawach \cite{MR4247793} on parametrized forms of dichotomies (originally due to Gowers \cite{MR1164759}) for the spaces $\mathrm{FIN}_k$, $\mathrm{FIN}_{\pm k}$, and the Banach space $c_0$. Calder\'on, Di Prisco, and Mijares \cite{MR4391477} have also recently developed local forms of the parametrized results for $\mathrm{FIN}_k$.

The main result of the present article is the following parametrized version of Theorem \ref{thm:local_Rosendal}:

\begin{thm}\label{thm:local_para_Rosendal}
	Let $\mathcal{H}$ be a strategic $(p^+)$-family on $E$. If $\mathbb{A}\subseteq \mathbb{R}\times{E^{[\infty]}}$ is analytic, then there is a nonempty perfect set $P\subseteq\mathbb{R}$ and an $X\in\mathcal{H}$ such that either:
	\begin{enumerate}[(1)]
		\item I has a strategy $\sigma$ in $F[X]$ such that $P\times[\sigma]\subseteq\mathbb{A}^c$, or
		\item for every $t\in P$, II has a strategy in $G[X]$ for playing into $\mathbb{A}_t$.
	\end{enumerate}
\end{thm}

Here, $\mathbb{A}_t=\{X\in{E^{[\infty]}}:(t,X)\in\mathbb{A}\}$ is the \emph{$t$-slice of $\mathbb{A}$} for $t\in\mathbb{R}$, and $[\sigma]$ the set of all outcomes of $F[X]$ wherein player I follows $\sigma$. We note that even in the non-localized case, where $\mathcal{H}=E^{[\infty]}$ and Theorem \ref{thm:local_para_Rosendal} becomes a parametrized form of Theorem \ref{thm:Rosendal}, our result is new. A major component of our proof will be the following analogue of Theorem \ref{thm:Baum_Laver}:

\begin{thm}\label{thm:Sacks_pres}
	If $\mathcal{F}$ is a strategic $(p^+)$-filter on $E$ and $g$ is $\mathbf{V}$-generic for $\mathbb{S}$, then $\mathcal{F}$ generates a strong $(p^+)$-filter in $\mathbf{V}[g]$
\end{thm}

We note some anomalies in our results: first, the conclusions for players I and II in Theorem \ref{thm:local_para_Rosendal} are asymmetric --- we will see in Section \ref{sec:perf_strat} that this is necessary and these conclusions are appropriately sharp. 
Second, in Theorem \ref{thm:local_para_Rosendal}, we make the additional assumption, not present in Theorem \ref{thm:local_Rosendal}, that $\mathcal{H}$ is strategic, and in Theorem \ref{thm:Sacks_pres} we do not assert that this property is preserved, only that the resulting filter in $\mathbf{V}[g]$ has the apparently weaker property of being a \emph{strong} $(p^+)$-filter (defined in Section \ref{sec:local_1d}). We do not know whether the assumption of ``strategic'' here can be weakened to ``strong'' in $\mathsf{ZFC}$. However, thanks to the anonymous referee, we show that ``strong $(p^+)$'' and ``strategic $(p^+)$'' coincide under modest large cardinal assumptions.

The rest of this article is organized as follows: Section \ref{sec:1d} begins by showing that the weak form of the pigeonhole principle in $(\dagger)$ can be parametrized by perfect sets (Theorem \ref{thm:para_1d}). We then apply this to obtain dichotomies for parametrized families of linear transformations on $E$ (Theorems \ref{thm:para_trans} and \ref{thm:para_trans_2}). Section \ref{sec:1d} is entirely self-contained and uses only basic set theory. We review the relevant definitions for families of block sequences from \cite{MR3864398} in Section \ref{sec:families}. Sacks forcing and the fusion method are reviewed in Section \ref{sec:Sacks}. In Section \ref{sec:local_1d}, we prove a local form (Theorem \ref{thm:local_para_1d}) of the parametrized weak pigeonhole principle from Section \ref{sec:1d} and then use it to establish our main preservation result, Theorem \ref{thm:Sacks_pres}. In Section \ref{sec:perf_strat}, we define perfectly strategically Ramsey sets and prove Theorem \ref{thm:local_para_Rosendal} for open sets (Lemma \ref{lem:local_para_Rosendal_open}). This is extended to all analytic sets in Section \ref{sec:analytic}, completing the proof of Theorem \ref{thm:local_para_Rosendal}. The relationship between the strategic and strong properties under large cardinal assumptions (Theorem \ref{thm:strong_strat}), and their preservation (Corollary \ref{cor:strat_pres}), are addressed in Section \ref{sec:strat_pres}. We conclude in Section \ref{sec:end} with a discussion of possible improvements of our results and methods.

Throughout, we will make use of standard descriptive set-theoretic facts about Borel and analytic subsets of Polish spaces, all of which can be found in \cite{MR1321597}. We will also use the method of forcing and absoluteness, for which standard references are \cite{MR1940513} and \cite{MR597342}. While these metamathematical techniques are employed in the proof of Theorem \ref{thm:local_para_Rosendal}, we stress that this result is, ultimately, established in $\mathsf{ZFC}$. A familiarity with infinite-dimensional Ramsey theory, as in \cite{MR2603812}, while helpful, is not necessary. Applications of our results to Banach spaces, including a parametrized form of Gowers' dichotomy, will be in a separate, forthcoming article \cite{Smythe:ParaRamseyBlockII}.

\section*{Acknowledgements}

The impetus for this article arose from conversations with Justin Moore and Stevo Todor\v{c}evi\'{c} dating back to 2016, about the author's thesis work \cite{Smythe_thesis}, later published in \cite{MR3864398}. Several email exchanges with Christian Rosendal and a visit to his Chicago office on a chilly February day in 2019 were invaluable. Various chats, correspondences, and MathOverflow exchanges with others over the intervening years, including Andreas Blass, Fran\c{c}ois Dorais, Edna Jones, Arno Pauly, No\'e de Rancourt, Richard Shore, and Yuan Yuan Zheng, have also contributed to this work. Finally, we must thank the anonymous referee for their comments and careful reading of this article, in particular for discovering an error and providing the material for Section \ref{sec:strat_pres}.

\section{Partitions and linear transformations}\label{sec:1d}

We begin by considering families of partitions of the underlying vector space and establish a parametrized form (Theorem \ref{thm:para_1d}) of the ``weak pigeonhole principle'' in ($\dagger$). We then apply this to parametrized families of linear transformations (Theorems \ref{thm:para_trans} and \ref{thm:para_trans_2}). In all of these results, the existence of a subspace satisfying the conclusion for a \emph{single} partition or linear transformation, respectively, is obvious. The novelty here lies in the fact that we obtain the same conclusion for perfectly many partitions or transformations simultaneously, all witnessed by the same subspace. We hope that these preliminary results suggest possible future applications, in addition to being interesting in their own right.

We say that a family $\{A_t:t\in\mathbb{R}\}$ of subsets of $E^\times$ is \emph{Borel} if the function $t\mapsto A_t$ is Borel measurable $\mathbb{R}\to 2^{E^\times}$, or equivalently, it is Borel when viewed as a subset of $\mathbb{R}\times 2^{E^\times}$.

\begin{thm}\label{thm:para_1d}
	For every Borel family $\{A_t:t\in\mathbb{R}\}$ of subsets of $E^\times$, there is a nonempty perfect set
	$P\subseteq\mathbb{R}$ and an infinite-dimensional subspace $V$ of $E$ such that either:
	\begin{enumerate}[(1)]
		\item for all $t\in P$, $A_t\cap V=\emptyset$, or
		\item for all $t\in P$, $A_t$ is asymptotic below $V$.
	\end{enumerate}
\end{thm}

Here, a set $A\subseteq E^\times$ is \emph{asymptotic below} a subspace $V$ (or block sequence $X$, respectively) if $A$ meets every infinite-dimensional subspace of $V$ (or $\langle X\rangle$). We proceed with two general lemmas, the first of which is standard, while the second was suggested by the anonymous referee.

\begin{lemma}\label{lem:omega_1_dense}
	Suppose that $(t_\alpha)_{\alpha<\omega_1}$ is an $\omega_1$-sequence of distinct elements of $\mathbb{R}$. Then, there is an $\alpha_0<\omega_1$ such that there are uncountably many $\beta,\gamma<\omega_1$ for which $t_\beta<t_{\alpha_0}<t_\gamma$.
\end{lemma}

\begin{proof}
	First, we claim that there is an $A\in[\omega_1]^{\omega_1}$ such that either for all $\alpha\in A$, there are uncountably many $\gamma\in A$ for which $t_\alpha<t_\gamma$, or for all $\alpha\in A$, there are uncountably many $\gamma\in A$ for which $t_\gamma<t_\alpha$. Suppose no $A\in[\omega_1]^{\omega_1}$ satisfies the latter. In particular, there is an $\alpha<\omega_1$ such that there are only countably many $\gamma<\omega_1$ for which $t_\gamma\leq t_{\alpha}$. Likewise, there is a $\beta$ in the interval $(\alpha,\omega_1)$ satisfying the same property. We can continue in this fashion to build an unbounded set $B\in[\omega_1]^{\omega_1}$ such that for all $\beta\in B$, there are only countably many $\gamma<\omega_1$ for which $t_\gamma\leq t_\beta$, and consequently, uncountably many $\gamma\in B$ for which $t_\beta<t_\gamma$, proving the claim.

	Thus, we may thin down our original sequence and assume that for all $\alpha<\omega_1$, there are uncountably many $\gamma<\omega_1$ for which $t_\alpha<t_\gamma$ (we can modify the argument accordingly for the other case). Now, if for some $\alpha<\omega_1$, there are uncountably many $\beta<\omega_1$ for which $t_\beta<t_\alpha$, we are done, so we suppose that for all $\alpha<\omega_1$, there are only countably many such $\beta$. Then, we may construct a strictly increasing subsequence of the $t_\alpha$'s in order-type $\omega_1$, using that at limit stages, there are only countably many $t_\alpha$'s below any of the countably many terms thus constructed. This, however, contradicts the well-known fact that $\omega_1$ does not order-embed into $\mathbb{R}$.
\end{proof}

\begin{lemma}\label{lem:ref}
	Suppose that $H_0,\ldots,H_n$ are uncountable pairwise disjoint sets, $H=\bigcup_{i\leq n} H_i$, and $Y_\alpha\subseteq E$ for $\alpha\in H$ are such that for any finite $J\subseteq H$, $\bigcap_{\alpha\in J}Y_\alpha\neq\emptyset$. Then, there exists uncountable subsets $K_i\subseteq H_i$ for each $i\leq n$ such that $\bigcap_{\alpha\in K}Y_\alpha\neq\emptyset$, where $K=\bigcup_{i\leq n} K_i$.
\end{lemma}

\begin{proof}
	Let $H'\subseteq H_0\times\cdots\times H_n$ be an uncountable set of tuples such that the projection maps $H'\to H_i$, for $i\leq n$, are injective, i.e., each element of $H_i$ occurs in at most one element of $H'$. For each $(\alpha_0,\ldots,\alpha_n)\in H'$, choose some $x_{(\alpha_0,\ldots,\alpha_n)}\in Y_{\alpha_0}\cap\cdots\cap Y_{\alpha_n}\neq\emptyset$. Since $H'$ is uncountable, but $E$ is countable, there is an uncountable $K'\subseteq H'$ such that for all $(\alpha_0,\ldots,\alpha_n)\in K'$, the $x_{(\alpha_0,\ldots,\alpha_n)}$ coincide, say are equal to $x$. For each $i\leq n$, let $K_i\subseteq H_i$ be the projection of $K'$ onto the $i$th coordinate and $K=\bigcup_{i\leq n}K_i$. Then, by our choice of $H'$, each $K_i$ is uncountable and $x\in\bigcap_{\alpha\in K}Y_\alpha\neq\emptyset$.
\end{proof}

The next crucial lemma is based on Lemma 3.1.1 in \cite{MR891249}.

\begin{lemma}\label{lem:vEMS}
	Suppose that $(X_\alpha)_{\alpha<\omega_1}$ and $(t_\alpha)_{\alpha<\omega_1}$ are $\omega_1$-sequences in $E^{[\infty]}$ and $\mathbb{R}$, respectively, such that for all $\alpha<\beta<\omega_1$, $X_\beta\preceq^* X_\alpha$ and $t_\alpha\neq t_\beta$. Then, there is a countable subset $\Sigma$ of $\omega_1$ and a $Z\in E^{[\infty]}$ such that:
		\begin{enumerate}[(i)]
			\item $(\{t_\alpha:\alpha\in\Sigma\},<)$ is order isomorphic to $(\mathbb{Q},<)$, and
			\item $Z\preceq X_\alpha$ for all $\alpha\in\Sigma$.
		\end{enumerate}
\end{lemma}

\begin{proof}
	Enumerate $\mathbb{Q}$ as $\{r_n:n\in\omega\}$. We will construct $\alpha_n\in\omega_1$, $A_n\in[\omega_1]^{\omega_1}$, and $z_n\in E^\times$, for $n\in\omega$, so that the mapping $r_n\mapsto t_{\alpha_n}$ is an order isomorphism, each interval in $\mathbb{R}$ determined by $t_{\alpha_i}$, for $i\leq n$, contains uncountably many $t_\alpha$ for $\alpha\in A_n$, and the $z_n$'s form a block sequence.
	
	We start by choosing $\alpha_0\in\omega_1$ as in Lemma \ref{lem:omega_1_dense} so that $H_0=\{\beta\in\omega_1:t_\beta<t_{\alpha_0}\}$ and $H_1=\{\gamma\in\omega_1:t_{\alpha_0}<t_\gamma\}$ are uncountable disjoint sets. Put $H=H_0\cup H_1$ and for each $\alpha\in H$, let $Y_\alpha=\langle X_\alpha\rangle\cap\langle X_{\alpha_0}\rangle\setminus\{0\}$. By Lemma \ref{lem:ref}, there are uncountable sets $K_0\subseteq H_0$ and $K_1\subseteq H_1$ for which $\bigcap_{\gamma\in K_0\cup K_1}Y_\gamma\neq\emptyset$. Put $A_0=K_0\cup K_1$ and choose $z_0\in\bigcap_{\gamma\in A_0}Y_\gamma$.
	
	Suppose we have defined $\alpha_n$, $A_n$, and $z_n$ as described. Choose an $\alpha_{n+1}>\alpha_n$ in $A_n$ so that the mapping $r_i\mapsto t_{\alpha_i}$, for $i\leq n+1$, is an order isomorphism, and each interval determined by the $t_{\alpha_i}$, for $i\leq n+1$, still contains uncountably many $t_\alpha$'s for $\alpha\in A_n$. This is again possible by applying Lemma \ref{lem:omega_1_dense} to the relevant intervals. For each of the now $n+2$ intervals determined by the $t_{\alpha_i}$, for $i\leq n+1$, let $H_j$ be the corresponding set of $\alpha\in A_n$ for which $t_\alpha$ lies in that interval. Put $H=\bigcup_{i\leq n+1}H_i$ and $Y_\alpha=(\langle X_\alpha/z_n\rangle\cap\bigcap_{i\leq n+1}\langle X_{\alpha_i}\rangle)\setminus\{0\}$ for each $\alpha\in H$. Applying Lemma \ref{lem:ref}, there are uncountable sets $K_j\subseteq H_j$, for $j\leq n+1$, and a nonzero vector $z_{n+1}\in\bigcap_{\gamma\in A_{n+1}}Y_\gamma$, where $A_{n+1}=\bigcup_{j\leq n+1} K_j$. This describes the construction.
	
	Let $\Sigma=\{\alpha_n:n\in\omega\}$ and $Z=(z_n)_{n\in\omega}$. Property (i) has clearly been ensured by construction. For property (ii), observe that for each $n\in\omega$, $z_n$ was chosen explicitly to be in $\bigcap_{i\leq n}\langle X_{\alpha_n}\rangle$, while the choice of $A_{n}$ and the fact that $\Sigma\setminus\{\alpha_0,\ldots,\alpha_n\}\subseteq A_{n}$ ensures that $z_n\in\langle X_\alpha\rangle$ for all $\alpha\in\Sigma\setminus\{\alpha_0,\ldots,\alpha_n\}$, proving the claim.
\end{proof}

\begin{proof}[Proof of Theorem \ref{thm:para_1d}]
	We may first shrink the set of indices of the Borel family $\{A_t:t\in\mathbb{R}\}$ to a nonempty perfect set $Q\subseteq\mathbb{R}$ on which the mapping $t\mapsto A_t$ is continuous. Suppose that there is no nonempty perfect set $P\subseteq Q$ and no $X\in E^{[\infty]}$ such that (2) holds for
	$P$ and $V=\langle X\rangle$. We will construct witnesses to (1). By our assumptions, the set
	\[
		C_0=\{t\in Q:A_t \text{ is asymptotic (for $E$)}\}
	\]
	contains no nonempty perfect subset.\footnote{It is tempting to use the perfect set property to conclude that $C_0$ is countable. However, this set is coanalytic, and the perfect set property is not provable for coanalytic sets in $\mathsf{ZFC}$. We do not know if $C_0$ is necessarily of lower complexity, i.e., whether it is Borel.} Consequently, its complement in $Q$ is uncountable, as the complement of any countable set must contain a perfect set, and in particular, nonempty. So, we may choose a $t_0\in Q\setminus C_0$ and an $X_0\in E^{[\infty]}$ such that $A_{t_0}\cap\langle X_0\rangle=\emptyset$. We will continue to construct $t_\alpha\in Q$ and $X_\alpha\in E^{[\infty]}$, for $\alpha<\omega_1$, such that all of the $t_\alpha$'s are distinct, $X_\beta\preceq^* X_\alpha$ whenever $\beta<\alpha$, and $A_{t_\alpha}\cap\langle X_\alpha\rangle=\emptyset$.
	
	Suppose that, for some $\gamma<\omega_1$, we have constructed $t_\alpha$ and $X_\alpha$ as described for all $\alpha<\gamma$. If $\gamma=\beta+1$, then consider the set
	\[
		C_\gamma=\{t\in Q:A_t\text{ is asymptotic for $X_\beta$}\}.
	\]
	By assumption, this set contains no nonempty perfect subset, so its complement in $Q$ is uncountable, and we may choose a $t_\gamma\in Q\setminus C_\gamma$ distinct from $t_\alpha$, for all $\alpha<\gamma$, and an $X_\gamma\preceq X_\beta$ such that $A_{t_\gamma}\cap \langle X_\gamma\rangle=\emptyset$. If instead, $\gamma$ is a limit ordinal, first diagonalize to obtain an $X\in E^{[\infty]}$ such that $X\preceq^* X_\alpha$ for all $\alpha<\gamma$. Then, argue as in the successor case to obtain a $t_\gamma$ distinct from $t_\alpha$, for all $\alpha<\gamma$, and an $X_\gamma\preceq X$ such that $A_{t_\gamma}\cap \langle X_\gamma\rangle=\emptyset$.	
	
	Let $\Sigma$ and $Z$ be as in Lemma \ref{lem:vEMS}, as applied to the sequences $(X_\alpha)_{\alpha<\omega_1}$ and $(t_\alpha)_{\alpha<\omega_1}$ constructed above. Let $P$ be a nonempty perfect subset of $\overline{\{t_\alpha:\alpha\in\Sigma\}}$ and $V=\langle Z\rangle$. 
 
	Take $t\in P$, say with $\lim_n t_{\alpha_n}= t$, where each $\alpha_n$ is in $\Sigma$. Towards a contradiction, suppose that $v\in A_t\cap V$, with $v\neq 0$. By continuity of the mapping $t\mapsto A_t$ on $Q$, for sufficiently large $n$, we have that $v\in A_{t_{\alpha_n}}$, but $Z\preceq X_{\alpha_n}$, so $v\in A_{t_{\alpha_n}}\cap\langle X_{\alpha_n}\rangle$, and thus $A_{t_{\alpha_n}}\cap\langle X_{\alpha_n}\rangle\neq\emptyset$, contrary to how $t_\alpha$ and $X_\alpha$ were chosen above. Thus, $A_t\cap V=\emptyset$ for all $t\in P$, verifying (1).
\end{proof}

Let $L(E)$ denote the set of all linear transformations from $E$ to itself. $L(E)$ inherits the structure of an uncountable Polish space when viewed as a subspace of $E^E$, or equivalently, when endowed with the topology of pointwise convergence. A family $\{T_t:t\in\mathbb{R}\}$ of transformations in $L(E)$ is again \emph{Borel} if the mapping $t\mapsto T_t$ is Borel measurable.

\begin{thm}\label{thm:para_trans}
	For every Borel family $\{T_t:t\in\mathbb{R}\}$ of linear transformations on $E$, there is a nonempty perfect set $P\subseteq \mathbb{R}$ and an infinite-dimensional subspace $V$ of $E$ such that either:
	\begin{enumerate}[(1)]
		\item for all $t\in P$, $V\subseteq\ker(T_t)$, or
		\item for all $t\in P$, $T_t\upharpoonright V$ is injective.
	\end{enumerate}
\end{thm}

\begin{proof}
	For each $t\in\mathbb{R}$, let
	\[
		A_t=\{v\in E^\times:T_t(v)\neq 0\}.
	\]
	Then, $\{A_t:t\in\mathbb{R}\}$ is a Borel family of subsets of $E^\times$. By Theorem \ref{thm:para_1d}, there is a nonempty perfect set $P\subseteq\mathbb{R}$ and an infinite-dimensional subspace $V$ of $E$ such that either (1$'$) for all $t\in P$, $A_t\cap V=\emptyset$, or (2$'$) for all $t\in P$, $A_t$ is asymptotic below $V$. If (1$'$) holds, then for all $t\in P$ and $v\in V$, $T_t(v)=0$, which proves (1). So, we may suppose that (2$'$) holds.
	
	For $t\in P$, if $T_t\upharpoonright V$ has infinite-dimensional kernel, then $A_t\cap\ker(T_t\upharpoonright V)=\emptyset$, contradicting that $A_t$ is asymptotic below $V$. Thus, $T_t\upharpoonright V$ has finite-dimensional kernel for each $t\in P$. For each $n\in\omega$, let 
	\[
		P_n=\{t\in P:\ker(T_t\upharpoonright V)\subseteq\langle e_0,\ldots,e_n\rangle\}.
	\] 
	Each $P_n$ is Borel and $P=\bigcup_{n\in\omega} P_n$, so since $P$ is uncountable, there is some $n_0\in\omega$ for which $P_{n_0}$ is uncountable and thus contains a nonempty perfect subset, say $Q$. Let $W=V/n_0$. It follows that for each $t\in Q$, $\ker(T_t\upharpoonright W)=\{0\}$, and so $T_t\upharpoonright W$ is injective.
\end{proof}

One might ask if we can improve (2) in Theorem \ref{thm:para_trans} to get a perfect family of transformations and an infinite-dimensional subspace $V$ such that all of the transformations restrict to isomorphisms from $V$ onto itself. This is not possible, as the following example shows.

\begin{example}
	Let $\mathcal{A}$ be a nonempty perfect \emph{almost disjoint} collection of infinite block sequences in $E$, i.e., for any distinct $X,Y\in\mathcal{A}$, $\langle X\rangle \cap \langle Y\rangle$ is finite-dimensional. Such a family exists by Proposition 2.2 of \cite{MR4045990}. For each $X=(x_n)_{n\in\omega}\in \mathcal{A}$, let $T_X:E\to X$ be the isomorphism induced by mapping $e_n\mapsto x_n$. If $X,Y\in\mathcal{A}$ are distinct and $V$ is any infinite-dimensional subspace, then $T_X\upharpoonright V$ and $T_Y\upharpoonright V$ have almost disjoint ranges. In particular, there is no nontrivial subfamily of the $T_X$'s such that every $T_X\upharpoonright V$ maps onto $V$.
\end{example}

We also have the following complementary result to Theorem \ref{thm:para_trans}.

\begin{thm}\label{thm:para_trans_2}
	For every Borel family $\{T_t:t\in\mathbb{R}\}$ of linear transformations on $E$, there is a nonempty perfect set $P\subseteq \mathbb{R}$ and an infinite-dimensional subspace $V$ of $E$ such that either:
	\begin{enumerate}[(1)]
		\item for all $t\in P$, $V\subseteq\mathrm{ran}(T_t)$, or
		\item for all $t\in P$, $\mathrm{ran}(T_t)\cap V=\{0\}$.
	\end{enumerate}		
\end{thm}

\begin{proof}
	As in the proof of Theorem \ref{thm:para_trans}, we apply Theorem \ref{thm:para_1d} to the Borel family
	\[
		A_t=\{v\in E^\times:v\notin\mathrm{ran}(T_t)\},
	\]	
	for $t\in\mathbb{R}$, to obtain a nonempty perfect $P\subseteq\mathbb{R}$ and an infinite-dimensional subspace $V$ of $E$ such that either (1$'$) for all $t\in P$, $A_t\cap V=\emptyset$, or (2$'$) for all $t\in P$, $A_t$ is asymptotic below $V$. Clearly, (1$'$) implies (1). If (2$'$) holds, then it follows that $\mathrm{ran}(T_t)\cap V$ is finite-dimensional for all $t\in P$, so we can argue as in the proof of Theorem \ref{thm:para_trans} to shrink $P$ to a nonempty perfect subset $Q$, and $V$ to an infinite-dimensional subspace $W$, for which (2) holds.
\end{proof}

\section{Families of block sequences}\label{sec:families}

Theorem \ref{thm:local_para_Rosendal} is phrased in terms of families of block sequences satisfying certain combinatorial properties originally isolated by the author in \cite{MR3864398}. These properties, which we review below, are analogous to selectivity of coideals on $\omega$ (see \cite{MR1644345} or \cite{MR0491197}).

\begin{defn}
	\begin{enumerate}[(a)]
		\item We say that $\mathcal{H}\subseteq{E^{[\infty]}}$ is a \emph{family} on $E$ if it is nonempty and closed upwards with respect to $\preceq^*$.
		\item A family $\mathcal{F}$ is a \emph{filter} on $E$ if for all $X,Y\in\mathcal{F}$, there is a $Z\in\mathcal{F}$ such that $Z\preceq X$ and $Z\preceq Y$.
	\end{enumerate}
\end{defn}

Given a family $\mathcal{H}$ and an $X\in\mathcal{H}$, we let $\mathcal{H}\upharpoonright X=\{Y\in\mathcal{H}:Y\preceq X\}$.

\begin{defn}
	Suppose $\mathcal{H}$ is a family on $E$ and $X\in\mathcal{H}$. A set $D\subseteq E$ is \emph{$\mathcal{H}$-dense below $X$} if for all $Y\in\mathcal{H}\upharpoonright X$, there is a $Z\preceq Y$ such that $\langle Z\rangle\subseteq D$. Likewise, a set $\mathbb{D}\subseteq E^{[\infty]}$ is \emph{$\mathcal{H}$-dense below $X$} if for all $Y\in\mathcal{H}\upharpoonright X$, there is a $Z\preceq Y$ such that $Z\in\mathbb{D}$. By just \emph{$\mathcal{H}$-dense}, we mean $\mathcal{H}$-dense below $(e_n)_{n\in\omega}$.
\end{defn}

Note that if $\mathcal{F}$ is a filter on $E$, then a set is $\mathcal{F}$-dense if and only if it is $\mathcal{F}$-dense below some $X\in\mathcal{F}$. The next definition specifies a certain largeness condition for families, playing the role of ``coideal'' in this setting and, in particular, the role of ``ultra'' for filters.

\begin{defn}
	A family $\mathcal{H}$ on $E$ is \emph{full} if whenever $D\subseteq E$ is $\mathcal{H}$-dense below some $X\in\mathcal{H}$, there is a $Z\in\mathcal{H}\upharpoonright X$ such that $\langle Z\rangle\subseteq D$.
\end{defn}

While the preceding definition was used in \cite{MR3864398}, a more enlightening equivalent condition is that full families are exactly those which witness a form of the weak pigeonhole principle ($\dagger$) described in the Introduction.

\begin{lemma}[cf.~Proposition 3.3 in \cite{MR4516184}]\label{lem:full_pigeon}
	A family $\mathcal{H}$ on $E$ is full if and only if for every $A\subseteq E^\times$ and $X\in\mathcal{H}$, there is a $Y\in\mathcal{H}\upharpoonright X$ such that either $A\cap\langle Y\rangle=\emptyset$ or $A$ is asymptotic below $Y$.
\end{lemma}

\begin{proof}
	The proof follows immediately from that fact that if $A=E^\times\setminus D$, then $D$ is $\mathcal{H}$-dense below some $X\in\mathcal{H}$ if and only if $A$ is not asymptotic below any $Y\in\mathcal{H}\upharpoonright X$.	
\end{proof}

Note that in this lemma, we may only assert that $A$ is asymptotic below \emph{some} $Y\in\mathcal{H}\upharpoonright X$, rather than the original $X$. This is unavoidable, as $A$ may be disjoint from some subspace of $X$ which does not contain any block sequence in $\mathcal{H}$.  When $\mathcal{F}$ is a filter, the criterion for fullness can be stated more simply: $\mathcal{F}$ is full if and only if for every $A\subseteq E^\times$, there is a $Y\in\mathcal{F}$ such that either $A\cap\langle Y\rangle=\emptyset$ or $A$ is asymptotic below $Y$. 

\begin{defn}
	A family $\mathcal{H}$ on $E$ is:
	\begin{enumerate}[(a)]
		\item a \emph{$(p)$-family} if whenever $(X_n)_{n\in\omega}$ is a $\preceq$-decreasing sequence in $\mathcal{H}$, there is a diagonalization of $(X_n)_{n\in\omega}$ in $\mathcal{H}$.
		\item a \emph{$(p^+)$-family} if it is a full $(p)$-family.
		\item \emph{strategic} if whenever $\alpha$ is a strategy for II in $G[X]$, for some $X\in\mathcal{H}$, there is an outcome of $\alpha$ in $\mathcal{H}$.
	\end{enumerate}
\end{defn}

We have remarked that $E^{[\infty]}$ is itself a $(p)$-family, and it is trivially both full and strategic. Less trivial examples, including strategic $(p^+)$-filters, can be constructed using additional set-theoretic assumptions such as the Continuum Hypothesis ($\mathsf{CH}$) or Martin's Axiom ($\mathsf{MA}$), or by forcing directly with $({E^{[\infty]}},\preceq)$, however, the existence of full filters in ${E^{[\infty]}}$ cannot be proved in $\mathsf{ZFC}$ alone; see Sections 5 and 6 of \cite{MR3864398}.

Fullness of a family $\mathcal{H}$ tells us that it ``meets'' all $\mathcal{H}$-dense subsets of $E$. An important consequence of Theorem \ref{thm:local_Rosendal} is that, for strategic $(p^+)$-families, we can extend this to all $\mathcal{H}$-dense downwards closed (with respect to $\preceq$) subsets of $E^{[\infty]}$, provided those subsets are analytic. This is a $\mathsf{ZFC}$ instance of the genericity of strategic $(p^+)$-families.

\begin{lemma}\label{lem:H_dense}
	Let $\mathcal{H}$ be a $(p^+)$-family on $E$ and $X\in\mathcal{H}$. If $\mathbb{D}\subseteq E^{[\infty]}$ is $\mathcal{H}$-dense below $X$, downwards closed, and analytic, then there is a $Y\in\mathcal{H}\upharpoonright X$ such that II has a strategy in $G[Y]$ for playing into $\mathbb{D}$. Consequently, if $\mathcal{H}$ is strategic, then $(\mathcal{H}\upharpoonright X)\cap\mathbb{D}\neq\emptyset$.
\end{lemma}

\begin{proof}
	By Theorem \ref{thm:local_Rosendal}, there is an $Y\in\mathcal{H}\upharpoonright X$ such that either (1) has a strategy in $F[Y]$ for playing out of $\mathbb{D}$ or (2) has a strategy in $G[Y]$ for playing into $\mathbb{D}$. Since $\mathbb{D}$ is $\mathcal{H}$-dense below $X$, there is a $Z\preceq Y$ such that $Z\in\mathbb{D}$. Then, in the game $F[Y]$, II may play in such a way that their moves always lie in $\langle Z\rangle$, and as $\mathbb{D}$ is downwards closed, the outcome will also lie in $\mathbb{D}$, contrary to (1). Thus, (2) holds, from which the claim follows.
\end{proof}

\section{Perfect sets and Sacks forcing}\label{sec:Sacks}

We view the set of all finite binary sequences $2^{<\omega}$ as a complete binary tree, ordered by initial segment $\sqsubseteq$, and $2^\omega$ as its set of infinite branches. A \emph{subtree} of $2^{<\omega}$ is a subset $p\subseteq 2^{<\omega}$ which is downwards closed with respect to $\sqsubseteq$. We say that $p$ is \emph{perfect} if for every $s\in p$, there are $t,u\in p$ such that $s\sqsubseteq t,u$ and which are incomparable with respect to $\sqsubseteq$. For such a $p$, we let $[p]$ denote the set of all infinite branches through $p$, a perfect subset of $2^\omega$. Conversely, every perfect subset of $2^\omega$ is of the form $[p]$ for some perfect subtree $p$ of $2^{<\omega}$.

\emph{Sacks forcing} is the set $\mathbb{S}$ of all nonempty perfect subtrees of $2^{<\omega}$, ordered by inclusion: $p\leq q$ if $p\subseteq q$, or equivalently, $[p]\subseteq[q]$. Our presentation of Sacks forcing follows \cite{MR794485}. For a general overview of Sacks forcing, see the survey \cite{MR2155534}.

For $p\in\mathbb{S}$ and $s\in p$, we let 
\[
	p|s=\{t\in p:t\sqsubseteq s\text{ or }s\sqsubseteq t\},
\] 
and note that $p|s\in\mathbb{S}$ and $p|s\leq p$. We say that $s\in p$ is a \emph{branching node} of $p$ if both $s^\smallfrown 0$ and $s^\smallfrown 1$ are in $p$. For $n\in\omega$, a node $s\in p$ has \emph{branching level $n$} if there are exactly $n$ branching nodes $t\in p$ such that $t\sqsubset s$. We let 
\[
	\ell(n,p)=\{s\in p:\text{ $s$ is a minimal node of branching level $n$}\}.
\] 
Note that $|\ell(n,p)|=2^n$ and $p=\bigcup_{s\in\ell(n,p)}p|s$ for each $n\in\omega$. 

We collect a few lemmas about Sacks forcing which will be used later on.

\begin{lemma}[Lemma 1.5 in \cite{MR794485}]\label{lem:compat_restr}
	Suppose $p\in\mathbb{S}$ and $n\in\omega$. If $q\leq p$, then there is an $s\in\ell(n,p)$ such that $q$ and $p|s$ are compatible.\qed
\end{lemma}

For $n\in\omega$, we write $p\leq_n q$ if $p\leq q$ and $\ell(n,p)=\ell(n,q)$. A sequence $(p_n)_{n\in\omega}$ of conditions in $\mathbb{S}$ is a \emph{fusion sequence} if $p_{n+1}\leq_n p_{n}$ for all $n\in\omega$. We call $p_\infty=\bigcap_{n\in\omega} p_n$ the \emph{fusion} of $(p_n)_{n\in\omega}$. The following important fact is known as the Fusion Lemma:

\begin{lemma}[Lemma 1.4 in \cite{MR794485}]\label{lem:fusion}
	If $(p_n)_{n\in\omega}$ is a fusion sequence in $\mathbb{S}$, then $p_\infty=\bigcap_{n\in\omega} p_n$ is a condition in $\mathbb{S}$ and $p_\infty\leq_n p_n$ for all $n\in\omega$.\qed
\end{lemma}

At the risk of imprecision, we say that a model $M$ is \emph{sufficient} if it is transitive and satisfies a large enough fragment of $\mathsf{ZFC}$ for whatever results are needed. Given such an $M$, if $G\subseteq\mathbb{S}\cap M$ is an $M$-generic filter for $\mathbb{S}\cap M$, then there is a unique $g\in\bigcap_{p\in G}[p]$, called an \emph{$M$-generic Sacks real}. As $G$ is definable from $g$ by $G=\{p\in\mathbb{S}\cap M:g\in[p]\}$, we have that $M[g]=M[G]$, see Lemma 1.1 of \cite{MR794485}. The next lemma is folklore, it is close to saying that $\mathbb{S}$ is proper as a notion of forcing, but we include the proof, as it is a typical application of fusion.

\begin{lemma}\label{lem:M_gen_Sacks}
	Let $M$ be a sufficient countable model. For any $p\in\mathbb{S}\cap M$, there is a $q\leq p$ such that for every $g\in[q]$, $g$ is $M$-generic for $\mathbb{S}\cap M$.
\end{lemma}

\begin{proof}
	Let $\{D_n:n\in\omega\}$ enumerate those subsets of $\mathbb{S}\cap M$ in $M$ which are dense open in $\mathbb{S}\cap M$. Choose $p_0\in D_0$ with $p_0\leq p$.
	
	Suppose we have chosen $p_0,\ldots,p_{n}\in\mathbb{S}\cap M$ with $p_{i+1}\leq_i p_{i}$ for each $i< n$. In order to find $p_{n+1}$, we refine $p_{n}$ cone-by-cone: enumerate $\ell(n,p_n)$ as $\{t_0,\ldots,t_{k}\}$. For each $\ell\leq k$, choose $q_{n+1}^\ell\in D_{n+1}$ with $q_{n+1}^\ell\leq p_{n}|t_\ell$. Let $p_{n+1}=\bigcup_{\ell\leq k}q_{n+1}^\ell$. Then, $p_{n+1}\in\mathbb{S}\cap M$ and $p_{n+1}\leq_{n} p_{n}$. 
	
	Let $q=\bigcap_{n\in\omega} p_n\leq p$, the fusion of $(p_n)_{n\in\omega}$. Let $g\in[q]$ and put $G=\{r\in\mathbb{S}\cap M:g\in[r]\}$, the corresponding filter. Since $g\in [p_0]$ and $p_0\in D_0$, we have that $G\cap D_0\neq\emptyset$. More generally, for each $n\in\omega$, $g\in [p_{n+1}]$, so there is some $\ell<|\ell(n,p_n)|$ for which $g\in[q_{n+1}^\ell]$. This implies $q_{n+1}^\ell\in G$, and since $q_{n+1}^\ell$ was chosen to be in $D_{n+1}$, we have that $G\cap D_{n+1}\neq\emptyset$. Hence, $g$ is $M$-generic for $\mathbb{S}\cap M$.
\end{proof}

Our final lemma of this section says that every sequence of ground model sets in a generic extension by $\mathbb{S}$ is covered by a sequence of finite sets, with prescribed size, from the ground model. This fact, known as the Sacks property, can also be proved using a routine fusion argument.

\begin{lemma}[Lemma 27 in \cite{MR2155534}]\label{lem:Sacks_prop}
	For any $p\in\mathbb{S}$, if $p\Vdash \dot{f}:\omega\to\mathbf{V}$, then there is a sequence of finite sets $(F_n)_{n\in\omega}\in\mathbf{V}$ such that $|F_n|\leq 2^n$ for each $n\in\omega$ and a $q\leq p$ such that $q\Vdash\forall n\in\omega(\dot{f}(n)\in F_n)$.\qed
\end{lemma}

\section{Localization and preservation}\label{sec:local_1d}

Our next result is a localized form of Theorem \ref{thm:para_1d}.

\begin{thm}\label{thm:local_para_1d}
	Let $\mathcal{H}$ be a strategic $(p^+)$-family on $E$. For every Borel family $\{A_t:t\in\mathbb{R}\}$ of subsets of $E^\times$, there is a nonempty perfect set $P\subseteq\mathbb{R}$ and an $X\in\mathcal{H}$ such that either:
	\begin{enumerate}[(1)]
		\item for all $t\in P$, $A_t\cap \langle X\rangle=\emptyset$, or
		\item for all $t\in P$, $A_t$ is asymptotic below $X$.
	\end{enumerate}
\end{thm}

\begin{proof}
	As in the proof of Theorem \ref{thm:para_1d}, we first shrink down to a nonempty perfect $Q\subseteq\mathbb{R}$ on which the mapping $t\mapsto A_t$ is continuous, and assume that there is no nonempty perfect set $P\subseteq Q$ and no $X\in\mathcal{H}$ such that (2) holds for $P$ and $X$. We will produce witnesses to (1).
	
	Consider the set
	\begin{align*}
		\mathbb{D}=\{Y\in E^{[\infty]}: \exists (t_n)_{n\in\omega}\in Q^\omega
		[&(\{t_n:n\in\omega\},<)\cong(\mathbb{Q},<)\land\\
		&\forall n\in\omega (A_{t_n}\cap\langle Y\rangle=\emptyset)]\}.
	\end{align*}
	Clearly, $\mathbb{D}$ is analytic and downwards closed. We will show that $\mathbb{D}$ is $\mathcal{H}$-dense. Once that is accomplished, we can apply Lemma \ref{lem:H_dense} to obtain a $Y\in\mathcal{H}\cap\mathbb{D}$ and a $(t_n)_{n\in\omega}\in Q^\omega$ such that $(\{t_n:n\in\omega\},<)\cong(\mathbb{Q},<)$  and $A_{t_n}\cap\langle Y\rangle=\emptyset$ for all $n\in\omega$. We can then argue exactly as in the last part of the proof of Theorem \ref{thm:para_1d} to obtain a perfect $P\subseteq Q$ such that for all $t\in P$, $A_t\cap\langle Y\rangle=\emptyset$, verifying (1).

	To show that $\mathbb{D}$ is $\mathcal{H}$-dense, fix an $X\in\mathcal{H}$. We must find a $Z\preceq X$ such that $Z\in\mathbb{D}$. First, consider the set
	\[
		\mathbb{D}_0=\{Y\preceq X:\exists t\in Q(A_t\cap\langle Y\rangle=\emptyset)\}.
	\]
	Again, $\mathbb{D}_0$ is analytic and downwards closed. We claim that $\mathbb{D}_0$ is $\mathcal{H}$-dense below $X$. If not, then there is an $X'\in\mathcal{H}\upharpoonright X$ such that for all $Y\preceq X'$ and all $t\in Q$, $A_t\cap\langle Y\rangle\neq\emptyset$, but this just says that (2) holds for $P=Q$ and $X'$, contrary to our assumptions. Thus, $\mathbb{D}_0$ is $\mathcal{\mathcal{H}}$-dense below $X$ and so by Lemma \ref{lem:H_dense}, $(\mathcal{H}\upharpoonright X)\cap\mathbb{D}_0\neq\emptyset$. That is, there is an $X_0\in\mathcal{H}\upharpoonright X$ and a $t_0\in Q$ such that $A_{t_0}\cap\langle X_0\rangle=\emptyset$. We continue in this fashion to construct $X_{\alpha}\in\mathcal{H}\upharpoonright X$ and $t_\alpha\in Q$, for $\alpha<\omega_1$, such that all of the $t_\alpha$'s are distinct, $X_{\alpha}\preceq^* X_\beta$ whenever $\beta<\alpha$, and $A_{t_\alpha}\cap\langle X_\alpha\rangle=\emptyset$.
	
	Suppose that, for some $\gamma<\omega_1$, we have constructed $X_\alpha$ and $t_\alpha$ for all $\alpha<\gamma$ as desired. Let $P_\gamma\subseteq Q$ be a nonempty perfect set which does not contain any of the previous countably many $t_\alpha$'s.  
	
	If $\gamma=\beta+1$, then consider the set
	\[
		\mathbb{D}_\gamma=\{Y\preceq X_\beta:\exists t\in P_\gamma(A_t\cap\langle Y\rangle=\emptyset)\}.
	\]
	As above, $\mathbb{D}_\gamma$ is analytic, downwards closed, and $\mathcal{H}$-dense below $X_\beta$; if not, then there is an $X'\in\mathcal{H}\upharpoonright X_\beta$ such that for all $t\in P_\gamma$, $A_t$ is asymptotic below $X'$, contrary to our assumptions. Thus, by Lemma \ref{lem:H_dense}, there is an $X_\gamma\in\mathcal{H}\upharpoonright X_\beta$ and a $t_\gamma\in P_\gamma$ such that $A_{t_\gamma}\cap\langle X_\gamma\rangle=\emptyset$. 
	
	If instead, $\gamma$ is a limit ordinal, first apply the $(p)$-property in $\mathcal{H}$ to obtain a $Y\in\mathcal{H}\upharpoonright X$ such that $Y\preceq^* X_\alpha$ for all $\alpha<\gamma$, then argue as in the successor case to obtain $X_\gamma\in\mathcal{H}\upharpoonright Y$ and $t_\gamma\in P_\gamma$ as required.
	
	We can now apply Lemma \ref{lem:vEMS} to the sequences $(X_\alpha)_{\alpha<\omega_1}$ and $(t_\alpha)_{\alpha<\omega_1}$ to obtain a $Z\preceq X$ such that $Z\in\mathbb{D}$, verifying that $\mathbb{D}$ is $\mathcal{H}$-dense, as claimed.
\end{proof}

Clearly, Theorems \ref{thm:para_trans} and \ref{thm:para_trans_2} can be localized in a similar fashion. We now turn to the proof of our main preservation result, Theorem \ref{thm:Sacks_pres}.

\begin{defn}
	A filter $\mathcal{F}$ on $E$ has the \emph{strong $(p)$-property} if whenever $(X_{\vec{x}})_{\vec{x}\in{E^{[<\infty]}}}$ is contained in $\mathcal{F}$, there is an $X\in\mathcal{F}$ such that $X/\vec{x}\preceq X_{\vec{x}}$ for all $\vec{x}\sqsubseteq X$.
\end{defn}

The strong $(p)$-property easily implies the $(p)$-property, and every strategic $(p^+)$-filter is a strong $(p^+)$-filter (Proposition 4.6 in \cite{MR3864398}). If $|F|<\infty$, then every $(p^+)$-filter is already a strong $(p^+)$-filter (Corollary 4.5 in \cite{MR4516184}), but we do not know if this holds in general. 

\begin{proof}[Proof of Theorem \ref{thm:Sacks_pres}]
	All uses of the forcing relation $\Vdash$ and names will be in reference to $\mathbb{S}$. Let $\overline{\mathcal{F}}$ be a name such that 
	\[
		\Vdash \text{$\overline{\mathcal{F}}$ is the $\preceq$-upwards closure of $\mathcal{F}$}.
	\]
	 Our goal, then, is to show that
	\[
		\Vdash\text{$\overline{\mathcal{F}}$ is a strong $(p^+)$-filter}.
	\]
	Note that $E$ and ${E^{[<\infty]}}$ remain unchanged by forcing, as is the fact that $\mathcal{F}$ generates a filter.
	
	First, we verify that the strong $(p)$-property is preserved. Suppose that $p\in\mathbb{S}$ is such that
	\[
		p\Vdash\forall\vec{x}\in E^{[<\infty]}(\dot{X}_{\vec{x}}\in\overline{\mathcal{F}}),
	\]
	where each $\dot{X}_{\vec{x}}$ is a name for an element of ${E^{[\infty]}}$. We can then find names $\dot{Y}_{\vec{x}}$, for each $\vec{x}\in{E^{[<\infty]}}$, such that
	\[
		p\Vdash\forall \vec{x}\in E^{[<\infty]}(\dot{Y}_{\vec{x}}\preceq\dot{X}_{\vec{x}}\land\dot{Y}_{\vec{x}}\in\mathcal{F}).
	\]
	
	By Lemma \ref{lem:Sacks_prop}, there is a family $\{F_{\vec{x}}:\vec{x}\in{E^{[<\infty]}}\}$ of finite subsets of $\mathcal{F}$ and a condition $q\leq p$ such that
	\[
		q\Vdash\forall \vec{x}\in E^{[<\infty]}(\dot{Y}_{\vec{x}}\in F_{\vec{x}}).
	\]
	Since $\mathcal{F}$ is a filter, for each $\vec{x}\in{E^{[<\infty]}}$, we can find a $Z_{\vec{x}}\in\mathcal{F}$ which is $\preceq$-below every element of $F_{\vec{x}}$. By the strong $(p)$-property, there is a $Z\in\mathcal{F}$ with $Z/{\vec{z}}\preceq Z_{\vec{z}}$ for all ${\vec{z}}\sqsubseteq Z$. But then,
	\[
		q\Vdash Z\in\overline{\mathcal{F}}\land\forall \vec{z}\sqsubseteq Z(Z/\vec{z}\preceq \dot{X}_{\vec{z}}).
	\]
	Thus, $\overline{\mathcal{F}}$ is forced to have the strong $(p)$-property.
	
	We now turn to verifying that $\overline{\mathcal{F}}$ is forced to be full. To this end, suppose that $p\in\mathbb{S}$ and $\dot{A}$ is a name such that $p\Vdash\dot{A}\subseteq E^\times$. For $v\in E^\times$, we say that a condition $q\leq p$ \emph{decides} ``$v\in \dot{A}$'' if $q\Vdash v\in\dot{A}$ or $q\Vdash v\notin\dot{A}$. Enumerate $E^\times$ as $\{v_n:n\in\omega\}$.
	
	\begin{claim}
		There is a $p_\infty\leq p$ such that for all $n\in\omega$ and all $s\in\ell(n,p_\infty)$, $p_\infty|s$ decides ``$v_n\in\dot{A}$''.
	\end{claim}
	
	\begin{proof}[Proof of claim.]
		We construct $p_\infty$ using fusion. Let $p_0\leq p$ decide ``$v_0\in\dot{A}$''. Say $\ell(1,p_0)=\{s,t\}$. Then, there are $q_s\leq p_0|s$ and $q_t\leq p_0|t$ which each decide ``$v_1\in\dot{D}$''. Let $p_1=q_s\cup q_t\leq_1 p_0$. Continue in this fashion, building a fusion sequence $(p_n)_{n\in\omega}$ such that for each $s\in\ell(n,p_n)$, $p_n|s$ decides ``$v_n\in\dot{A}$''. Then, $p_\infty=\bigcap_{n\in\omega} p_n$ is as claimed.
	\end{proof}
	
	For each $f\in[p_\infty]$, let
	\[
		A_f=\{v\in E^\times:\forall n\in\omega[v=v_n \rightarrow \exists s\in\ell(n,p_\infty)(s\sqsubset f\land p_\infty|s\Vdash v_n\in\dot{A})]\}.
	\]
	The assignment $[p_\infty]\to 2^{E^\times}:f\mapsto A_f$ is continuous. Thus, by Theorem \ref{thm:local_para_1d}, there is a $q\leq p_\infty$ and an $X\in\mathcal{F}$ such that either (1) for all $f\in [q]$, $A_f\cap \langle X\rangle=\emptyset$, or (2) for all $f\in [q]$, $A_f$ is asymptotic below $X$.
	
	Suppose that (1) holds, but $q\not\Vdash \dot{A}\cap\langle X\rangle=\emptyset$. Then, there is an $r\leq q$ and a $v=v_n\in \langle X\rangle$ such that $r\Vdash v_n\in\dot{A}$. Take $s\in r\cap\ell(n,p_\infty)$. Then, $r|s\leq p_\infty|s$, so by the choice of $p_\infty$, we must have that $p_\infty|s\Vdash v_n\in\dot{A}$ as well. But then, for $f\in[r]\subseteq[q]$ with $s\sqsubset f$, we have that $v_n\in A_f\cap\langle X\rangle$, contrary to (1). So, in this case, we must have that $q\Vdash\dot{A}\cap\langle X\rangle=\emptyset$.
	
	Next, suppose that (2) holds, but $q\not\Vdash\text{$\dot{A}$ is asymptotic below $X$}$. Then, there is an $r\leq q$ and a name $\dot{Z}$ such that $r\Vdash\dot{Z}\preceq X\land \dot{A}\cap\langle \dot{Z}\rangle=\emptyset$. As in the claim above, we can find an $r_\infty\leq r$ such that for all $n\in\omega$ and all $s\in\ell(n,r_\infty)$, $r_\infty|s$ decides ``$v_n\in\langle\dot{Z}\rangle$''.
	
	Since $r_\infty\Vdash\text{$\dot{Z}$ is an infinite block sequence}$, there is an $n_0\in\omega$ such that for some $t_0\in\ell(n_0,r_\infty)$, $r_\infty|t_0\Vdash v_{n_0}\in\langle\dot{Z}\rangle$. Likewise, there is an $n_1>n_0$ such that $\max(\mathrm{supp}(v_{n_0}))<\min(\mathrm{supp}(v_{n_1}))$ and for some $t_1\in\ell(n_1,r_\infty)$ extending $t_0$, $r_\infty|t_1\Vdash v_{n_1}\in\langle\dot{Z}\rangle$. Continuing in this fashion, we may construct a $\sqsubset$-increasing sequence $(t_k)_{k\in\omega}$ in $r_\infty$ and a block sequence $(v_{n_k})_{k\in\omega}\preceq X$ such that for each $k\in\omega$, $t_k\in\ell(n_k,r_\infty)$ and $r_\infty|t_k\Vdash v_{n_k}\in\langle\dot{Z}\rangle$. 
	
	Let $f=\bigcup_{k\in\omega}t_k\in[r_\infty]\subseteq[q]$. By (2), there is a $v=v_N\in A_f\cap\langle (v_{n_k})_{k\in\omega}\rangle$. Say $v_N\in\langle v_{n_0},v_{n_1},\ldots,v_{n_\ell}\rangle$. Take $s\in\ell(N,p_\infty)$ such that $s\sqsubset f$. Then, $p_\infty|s\Vdash v_N\in\dot{A}$. If $s\sqsubseteq t_{n_\ell}$, then $r_\infty|t_{n_\ell}\leq p_\infty|t_{n_\ell}\leq p_\infty|s$, so $r_\infty|t_{n_\ell}\Vdash v_N\in\dot{A}\cap\langle\dot{Z}\rangle$, while if $t_{n_\ell}\sqsubseteq s$, then $r_\infty|s\leq r_\infty|t_{n_\ell}$ and so again $r_\infty|s\Vdash v_N\in\dot{A}\cap\langle \dot{Z}\rangle$. Either way, this contradicts the fact that $r_\infty\leq r$ and $r\Vdash\dot{A}\cap\langle\dot{Z}\rangle=\emptyset$. Hence, $q\Vdash\text{$\dot{A}$ is asymptotic below $X$}$.
	
	As either (1) or (2) must hold, we have proven that either $q\Vdash\dot{A}\cap\langle X\rangle=\emptyset$ or $q\Vdash\text{$\dot{A}$ is asymptotic below $X$}$. Thus, by Lemma \ref{lem:full_pigeon} (more specifically, the formulation for filters given following its proof), $\overline{\mathcal{F}}$ is forced to be full.
\end{proof}

\section{Perfectly strategically Ramsey sets}\label{sec:perf_strat}

In looking for an appropriate parametrized form of Theorem \ref{thm:Rosendal}, one might initially arrive at the following conjecture:

\begin{quote}
${(\ast)}$ \emph{If $\mathbb{A}\subseteq \mathbb{R}\times{E^{[\infty]}}$ is analytic, then there is a nonempty perfect set $P\subseteq \mathbb{R}$ and an $X\in{E^{[\infty]}}$ such that either:
	\begin{enumerate}[(1)]
		\item I has a strategy $\sigma$ in $F[X]$ such that $P\times[\sigma]\subseteq\mathbb{A}^c$, or
		\item II has a strategy $\alpha$ in $G[X]$ such that $P\times[\alpha]\subseteq\mathbb{A}$.
	\end{enumerate}}
\end{quote}

\noindent Here, $[\sigma]$ denotes the set of all outcomes of $F[X]$ when I follows $\sigma$, and likewise for $[\alpha]$ in $G[X]$. Unfortunately, unless $|F|=2$ (where even stronger dichotomies hold, cf.~\cite{MR3717938}), $(\ast)$ is false:

\begin{example}\label{ex:no_unif_II}
	Assume $|F|>2$. Let $A_0$ and $A_1$ be an asymptotic pair, as in Example \ref{ex:asym_pair}. Consider the set
	\[
		\mathbb{A}=\{(f,(y_n)_{n\in\omega})\in2^\omega\times{E^{[\infty]}}:\forall n\in\omega(y_n\in A_{f(n)})\}.
	\]
	Note that we have replaced $\mathbb{R}$ with $2^\omega$ here, which we may do without loss of generality, and $\mathbb{A}$ is closed in $2^\omega\times E^{[\infty]}$. Suppose that $P\subseteq 2^\omega$ is a nonempty perfect set and $X\in{E^{[\infty]}}$. Then, for any $f\in P$, II can ensure that their moves $y_n$ in $F[X]$ satisfy $y_n\in A_{f(n)}$ for all $n\in\omega$, and thus $(f,(y_n)_{n\in\omega})\in\mathbb{A}$. In particular, there can be no strategy $\sigma$ for I in $F[X]$ for which $P\times[\sigma]\subseteq\mathbb{A}^c$.
	
	Suppose instead that $\alpha$ is a strategy for II in $G[X]$ such that $P\times[\alpha]\subseteq \mathbb{A}$. Let $f,g\in P$ be such that $f\neq g$ and $(y_n)_{n\in\omega}$ an outcome of $\alpha$. Then, for all $n\in\omega$, $y_n\in A_{f(n)}$ and $y_n\in A_{g(n)}$, but this contradicts the disjointness of $A_0$ and $A_1$. Thus, there is no such strategy $\alpha$ either. Note that we have only used that $|P|>1$ here.
\end{example}

The problem with $(\ast)$ is the uniformity in the parameter from $P$: (2) says that there is a single strategy $\alpha$ for II in $G[X]$ such that for all $f\in P$, every outcome of $\alpha$ lies in the slice $\mathbb{A}_f$. It is in dropping this uniformity, in effect, swapping the order of quantifiers in (2), that we find the correct parametrized form of Theorems \ref{thm:Rosendal} and \ref{thm:local_Rosendal}, our Theorem \ref{thm:local_para_Rosendal}. Note that in Example \ref{ex:no_unif_II}, for every $f\in 2^\omega$, II clearly has a strategy in $G[X]$ to play into $\mathbb{A}_f$. However, we are able to maintain the uniformity in (1), which we can see directly in the open case (cf.~Theorem 35.32 and Exercise 35.33 in \cite{MR1321597}).

\begin{lemma}\label{lem:strat_unif}
	Let $\mathbb{A}\subseteq 2^\omega\times{E^{[\infty]}}$ be open, $P\subseteq 2^\omega$ a nonempty perfect set, and $X\in{E^{[\infty]}}$. If for all $f\in P$, I has a strategy in $F[X]$ for playing out of $\mathbb{A}_f$, then there is a nonempty perfect $Q\subseteq P$ and a strategy $\sigma$ for $I$ in $F[X]$ such that $Q\times[\sigma]\subseteq\mathbb{A}^c$.	
\end{lemma}

\begin{proof}
	We can view strategies for I in $F[X]$ as functions from the set $X^{[<\infty]}$, of all finite block sequences in $X$, to $\omega$. Let $\mathcal{S}$ denote the space of all such functions, which we topologize as the product $\omega^{X^{[<\infty]}}$. Since $\mathbb{A}$ is open, we can fix a countable set $\mathcal{A}\subseteq 2^{<\omega}\times E^{[<\infty]}$ such that for all $(f,Y)\in 2^\omega\times E^{[\infty]}$, $(f,Y)\in \mathbb{A}$ if and only if there exists $(t,\vec{x})\in\mathcal{A}$ such that $t\sqsubset f$ and $\vec{x}\sqsubset Y$.
	
	Consider the set $\mathcal{B}\subseteq P\times\mathcal{S}$ defined by
	\[
		\mathcal{B}=\{(f,\sigma)\in P\times\mathcal{S}:\sigma \text{ is a strategy in $F[X]$ for playing out of $\mathbb{A}_f$}\}.
	\]
	So, $(f,\sigma)\in\mathcal{B}$ if and only if for every finite sequence $\vec{y}=(y_0,\ldots,y_n)$ of moves by II in $F[X]$ against $\sigma$, there is no $(t,\vec{x})\in\mathcal{A}$ such that $t\sqsubset f$ and $\vec{x}\sqsubset\vec{y}$. Thus, $\mathcal{B}$ is Borel. By the Jankov--von Neumann Uniformization Theorem (Theorem 18.1 in \cite{MR1321597}), $\mathcal{B}$ admits a $\sigma(\mathbf{\Sigma}^1_1)$-uniformization. That is, there is a $\sigma(\mathbf{\Sigma}^1_1)$-measurable function $f\mapsto\sigma_f$ on $P$ such that $(f,\sigma_f)\in\mathcal{B}$ for all $f\in P$. 
	
	As $\sigma(\mathbf{\Sigma}^1_1)$-measurable functions are Baire measurable, there is a nonempty perfect set $Q\subseteq P$ on which $f\mapsto\sigma_f$ is continuous (cf.~Theorem 8.38 in \cite{MR1321597}). In particular, the image of the compact set $Q$ under this map is a compact subset of $\mathcal{S}$. It follows that for any $\vec{x}\in X^{[<\infty]}$, the set $\{\sigma_f(\vec{x}):f\in Q\}$ is finite, say with maximum value $n_{\vec{x}}$. Let $\sigma$ be the strategy for $I$ in $F[X]$ defined by $\sigma(\vec{x})=n_{\vec{x}}$ for all $\vec{x}\in X^{[<\infty]}$. Then, for all $f\in Q$, any outcome in $F[X]$ against $\sigma$ is also a valid outcome against $\sigma_f$, and so $\sigma$ is a strategy for playing out of $\mathbb{A}_f$. Thus, $Q\times[\sigma]\subseteq\mathbb{A}^c$.
\end{proof}

In \cite{MR3864398}, given a family $\mathcal{H}$ on $E$, we defined a set $\mathbb{A}\subseteq E^{[\infty]}$ to be \emph{$\mathcal{H}$-strategically Ramsey} if for every $\vec{x}\in E^{[<\infty]}$ and $X\in\mathcal{H}$, there is a $Y\in\mathcal{H}\upharpoonright X$ such that either:
\begin{enumerate}[(1)]
	\item I has a strategy in $F[\vec{x},Y]$ for playing out of $\mathbb{A}$, or
	\item II has a strategy in $G[\vec{x},Y]$ for playing in to $\mathbb{A}$.
\end{enumerate}
Here, the games $F[\vec{x},X]$ and $G[\vec{x},X]$ are defined in exactly the same way as $F[X]$ and $G[X]$ except that we demand II's moves $(y_n)_{n\in\omega}$ to be supported above $\vec{x}$ and declare the outcome to be $\vec{x}^\smallfrown(y_n)_{n\in\omega}$. When $\mathcal{H}=E^{[\infty]}$, we just say that $\mathbb{A}$ is \emph{strategically Ramsey}, as in \cite{MR2604856}. Theorems \ref{thm:Rosendal} and \ref{thm:local_Rosendal} follow once it is shown that analytic sets are $\mathcal{H}$-strategically Ramsey, whenever $\mathcal{H}$ is a $(p^+)$-family. The following definition is the parametrized form of being $\mathcal{H}$-strategically Ramsey.

\begin{defn}
	Let $\mathcal{H}$ be a family on $E$. A set $\mathbb{A}\subseteq 2^\omega\times{E^{[\infty]}}$ is \emph{perfectly $\mathcal{H}$-strategically Ramsey} if for every $\vec{x}\in{E^{[<\infty]}}$, $X\in\mathcal{H}$, and $p\in\mathbb{S}$, there is a $Y\in\mathcal{H}\upharpoonright X$ and $q\leq p$ such that either:
	\begin{enumerate}[(1)]
		\item I has a strategy $\sigma$ in $F[\vec{x},Y]$ such that $[q]\times[\sigma]\subseteq\mathbb{A}^c$, or
		\item for every $f\in [q]$, II has a strategy in $G[\vec{x},Y]$ for playing into $\mathbb{A}_{f}$.	
	\end{enumerate}	
	When $\mathcal{H}={E^{[\infty]}}$, we just say that $\mathbb{A}$ is \emph{perfectly strategically Ramsey}.
\end{defn}

Theorem \ref{thm:local_para_Rosendal} will thus follow once we have proved that all analytic subsets of $2^\omega\times{E^{[\infty]}}$ are perfectly $\mathcal{H}$-strategically Ramsey, whenever $\mathcal{H}$ is a strategic $(p^+)$-family (Theorem \ref{thm:analytic} below).

It will be useful to ``finitize'' the Gowers game, as in \cite{MR1839387}: given $X\in E^{[\infty]}$, let $G^{<\infty}[X]$ be the two player game where I starts by playing a nonzero vector $x_0^{(0)}\in\langle X\rangle$, and then II responds with either $0$ or a nonzero vector $y_0\in\langle x_0^{(0)}\rangle$. If II plays $0$, then I plays a nonzero vector $x_1^{(0)}\in\langle X\rangle$ supported above $x_0^{(0)}$, and II must respond with either $0$ or a nonzero vector $y_0\in\langle x_0^{(0)},x_1^{(0)}\rangle$. If II plays a nonzero vector, the game restarts, with I playing a nonzero vector $x_0^{(1)}\in\langle X\rangle$, and so on. The nonzero plays of II are, again, required to form a block sequence $(y_n)_{n\in\omega}$, which is the outcome of the game. 
\[
	\begin{matrix}
		\mathrm{I}  & x^{(0)}_0,\ldots,x^{(0)}_{n_0} &       & x^{(1)}_0,\ldots,x^{(1)}_{n_1} &       & \cdots\\
		\mathrm{II} &	   & y_0\in\langle x^{(0)}_0,\ldots,x^{(0)}_{n_0}\rangle &       & y_1\in\langle x^{(1)}_0,\ldots,x^{(1)}_{n_1}\rangle &  \cdots
	\end{matrix}
\]
In essence, I is playing a block sequence $X_n\preceq X$ one vector at a time and II is responding with a nonzero vector $y_n\in\langle X_n\rangle$ as soon as enough of $X_n$ has been revealed to generate $y_n$. The game $G^{<\infty}[\vec{x},X]$, for $\vec{x}\in E^{[<\infty]}$, is defined similarly.

From the point of view of descriptive complexity, $G^{<\infty}[X]$ is a simpler game than $G[X]$ as both players play vectors, which can be coded as natural numbers, rather than infinite sequences of vectors. However, the following says we lose no generality in passing between $G[X]$ and $G^{<\infty}[X]$.

\begin{lemma}[Theorems 1.1 and 1.2 in \cite{MR1839387}]\label{lem:fin_Gowers}
	For any $\vec{x}\in E^{[<\infty]}$ and $X\in E^{[\infty]}$, the games $G[\vec{x},X]$ and $G^{<\infty}[\vec{x},X]$ are equivalent in the sense that for any $\mathbb{A}\subseteq E^{[\infty]}$, I (II, respectively) has a strategy in $G[\vec{x},X]$ to play into $\mathbb{A}$ if and only if I (II) has a strategy in $G^{<\infty}[\vec{x},X]$ to play into $\mathbb{A}$.
\end{lemma}

Our next lemma concerns the complexity of certain sets related to the games $F[\vec{x},X]$ and $G[\vec{x},X]$ (cf.~Section 20.D of \cite{MR1321597}).	Given $\mathbb{A}\subseteq 2^\omega\times E^{[\infty]}$ and $\vec{x}\in{E^{[<\infty]}}$, we define:
	\begin{align*}
		\mathbb{A}_{\vec{x},\mathrm{I}}=\{(f,X)\in 2^\omega\times{E^{[\infty]}}: &\text{ I has a strategy in $F[\vec{x},X]$}\\
		&\text{for playing out of $\mathbb{A}_f$}\},\\
		\mathbb{A}_{\vec{x},\mathrm{II}}=\{(f,X)\in 2^\omega\times {E^{[\infty]}}:&\text{ II has a strategy in $G[\vec{x},X]$}\\&\text{for playing into $\mathbb{A}_f$}\}.
	\end{align*}

\begin{lemma}\label{lem:strat_complexity}
	If $\mathbb{A}$ is open, then $\mathbb{A}_{\vec{x},\mathrm{I}}$ is analytic and $\mathbb{A}_{\vec{x},\mathrm{II}}$ is $\mathbf{\Sigma}^1_2$.
\end{lemma}

\begin{proof}
	For $\mathbb{A}_{\vec{x},\mathrm{I}}$, $(f,X)\in \mathbb{A}_{\vec{x},\mathrm{I}}$ if and only if there is a strategy $\sigma:X^{[<\infty]}\to\omega$ for I in $F[\vec{x},X]$ such that $(f,\sigma)$ is in the Borel set $\mathcal{B}$ defined (uniformly in $X$) in the proof of Lemma \ref{lem:strat_unif} above, modified for $\vec{x}$, and so $\mathbb{A}_{\vec{x},\mathrm{I}}$ is analytic.
	
	For $\mathbb{A}_{\vec{x},\mathrm{II}}$, by Lemma \ref{lem:fin_Gowers} we may pass to the finitized Gowers game $G^{<\infty}[\vec{x},X]$ and so $(f,X)\in\mathbb{A}_{\vec{x},\mathrm{II}}$ if and only if there is a function $\alpha:X^{<\infty}\to\langle X\rangle$ such that for any sequence $(x_n)_{n\in\omega}$ of moves by I in $G^{<\infty}[\vec{x},X]$ against $\alpha$, there are infinitely many $n\in\omega$ for which $\alpha(x_0,\ldots,x_n)\neq 0$, whenever $\alpha(x_0,\ldots,x_n)\neq 0$ and $m$ is the least value $>n$ such that $\alpha(x_0,\ldots,x_m)\neq 0$, then $\alpha(x_0,\ldots,x_n)\in\langle x_{n+1},\ldots,x_m\rangle$, and the sequence of nonzero values of $\alpha$ played in response to $(x_n)_{n\in\omega}$ must lie in $\mathbb{A}_f$. This is a $\mathbf{\Sigma}^1_2$ condition.
\end{proof}

We can now use metamathematical techniques to prove that open sets are perfectly $\mathcal{H}$-strategically Ramsey.

\begin{lemma}\label{lem:perf_F-strat_Ramsey_open}
	Let $\mathcal{F}$ be a strategic $(p^+)$-filter on $E$. Then, open sets are perfectly $\mathcal{F}$-strategically Ramsey.
\end{lemma}

\begin{proof}
	Let $\mathbb{A}\subseteq 2^\omega\times{E^{[\infty]}}$ be open, and fix $\vec{x}\in{E^{[<\infty]}}$, $X\in{E^{[\infty]}}$ and $p\in\mathbb{S}$. By Lemma \ref{lem:strat_complexity} above and Theorem 25.4 in \cite{MR1940513}, the statements $(f,Y)\in\mathbb{A}_{\vec{x},\mathrm{I}}$ and $(f,Y)\in\mathbb{A}_{\vec{x},\mathrm{II}}$ are upwards absolute for sufficiently rich transitive models.
		
	Let $M$ be the transitive collapse of a countable elementary submodel of $H(\theta)$, for some sufficiently large $\theta$, having an extra predicate for $\mathcal{F}$, and containing $X$ and $p$. By coding and elementarity, we can demand that $\mathcal{F}\cap M\in M$ and that
	\[
		M\models\text{$\mathcal{F}\cap M$ is a strategic $(p^+)$-filter}.
	\]
	Note that $M$ is sufficient, in the terminology of Lemma \ref{lem:M_gen_Sacks}.
	
	Suppose that $g$ is $M$-generic for $\mathbb{S}\cap M$ with $g\in[p]$. We work in $M[g]$. Let $\overline{\mathcal{F}}$ be the upwards closure of $\mathcal{F}\cap M$. By Theorem \ref{thm:Sacks_pres}, $\overline{\mathcal{F}}$ is a $(p^+)$-filter, and so Theorem \ref{thm:local_Rosendal} holds for it, applied to the (coded) open set $\mathbb{A}_g^{M[g]}$ (really, we are using that $\mathbb{A}_g^{M[g]}$ is $\overline{\mathcal{F}}$-strategically Ramsey in $M[g]$). Thus, there is some $Y\in\overline{\mathcal{F}}\upharpoonright X$ witnessing this. Passing to a block subsequence of $Y$ which lies in $M$, and a $p'\leq p$ in $\mathbb{S}\cap M$ which decides it, we may assume that $Y\in\mathcal{F}\cap M$ and
	\[
		p'\Vdash (g,Y)\in\mathbb{A}_{\vec{x},\mathrm{I}} \text{ or }p'\Vdash (g,Y)\in\mathbb{A}_{\vec{x},\mathrm{II}}
	\]
	
	It follows by absoluteness that either for all $M$-generic $g\in[p']$, $(g,Y)\in\mathbb{A}_{\vec{x},\mathrm{I}}$, or for all $M$-generic $g\in[p']$, $(g,Y)\in\mathbb{A}_{\vec{x},\mathrm{II}}$. Taking $q\leq p'$ as in Lemma \ref{lem:M_gen_Sacks} and applying Lemma \ref{lem:strat_unif} if necessary, $Y$ and $q$ witness that $\mathbb{A}$ is perfectly $\mathcal{F}$-strategically Ramsey.
\end{proof}

\begin{lemma}\label{lem:local_para_Rosendal_open}
	Let $\mathcal{H}$ be a strategic $(p^+)$-family on $E$. Then, open sets are perfectly $\mathcal{H}$-strategically Ramsey.
\end{lemma}

\begin{proof}
	Let $\mathbb{A}\subseteq 2^\omega\times E^{[\infty]}$ be open, and fix $\vec{x}\in E^{[<\infty]}$, $X\in\mathcal{H}$, and $p\in\mathbb{S}$. By Lemma 5.4 in \cite{MR3864398}, we can force with $(\mathcal{H},\preceq)$ to generically add a strategic $(p^+)$-filter $\mathcal{G}\subseteq\mathcal{H}$ containing $X$, and doing so does not add new reals (and thus no new block sequences, perfect subsets of $2^\omega$, Borel sets, etc). In $\mathbf{V}[\mathcal{G}]$, $\mathbb{A}$ is perfectly $\mathcal{G}$-strategically Ramsey by Lemma \ref{lem:perf_F-strat_Ramsey_open}, so there are $q\leq p$ and $Y\in\mathcal{G}\upharpoonright X\subseteq\mathcal{H}\upharpoonright X$ witnessing this. As we have added no new reals, $q$ and $Y$ are in $\mathbf{V}$, and the fact that they witness the relevant dichotomy involves only quantification over reals, and thus is also true in $\mathbf{V}$. Here we are implicitly using Lemma \ref{lem:fin_Gowers} to pass between $G[\vec{x},Y]$ to $G^{<\infty}[\vec{x},Y]$. Hence, $\mathbb{A}$ is perfectly $\mathcal{H}$-strategically Ramsey.
\end{proof}

\section{Extending to analytic sets}\label{sec:analytic}

To complete the proof of Theorem \ref{thm:local_para_Rosendal}, we will extend Lemma \ref{lem:local_para_Rosendal_open} to all analytic sets using a combinatorial forcing argument, similar to that used in the proof of Theorem \ref{thm:Rosendal} in \cite{MR2604856}.

For the remainder of this section, we fix a strategic $(p^+)$-family $\mathcal{H}$ on $E$. We reserve the notation $(p,X)$, $(q,Y)$, etc, for pairs in $\mathbb{S}\times\mathcal{H}$, and write $(q,Y)\leq(p,X)$ when $q\leq p$ and $Y\preceq X$. We use $\vec{x}$, $\vec{y}$, etc, for finite block sequences in $E^{[<\infty]}$, and $n$, $m$, etc, for natural numbers. 

The following definition and Lemmas \ref{lem:accept_reject} through \ref{lem:accept_play_reject} refer implicitly to sets $\mathbb{A}_n\subseteq 2^\omega\times{E^{[\infty]}}$, for $n\in\omega$, and $\mathbb{A}=\bigcup_{n\in\omega}\mathbb{A}_n$.

\begin{defn}
	Given $(p,X)$, $\vec{x}$, and $n$, we say that:
	\begin{enumerate}[(i)]
		\item $(\vec{x},n)$ \emph{accepts} $(p,X)$ if I has a strategy $\sigma$ in $F[\vec{x},X]$ such that $[p]\times[\sigma]\subseteq\mathbb{A}_n^c$.
		\item $(\vec{x},n)$ \emph{rejects} $(p,X)$ if no $(q,Y)\leq(p,X)$ is accepted by $(\vec{x},n)$.
		\item $(\vec{x},n)$ \emph{decides} $(p,X)$ if it either accepts or rejects $(p,X)$.
	\end{enumerate}
\end{defn}

\begin{lemma}\label{lem:accept_reject}
	\begin{enumerate}[(a)]
		\item For every $(p,X)$, $\vec{x}$, and $n$, there is a $(q,Y)\leq(p,X)$ such that $(\vec{x},n)$ decides $(q,Y)$.
		\item For every $(p,X)$, $\vec{x}$, and $n$, if $(\vec{x},n)$ accepts (rejects, respectively) $(p,X)$, $q\leq p$ and $Y\preceq^* X$ in $\mathcal{H}$, then $(\vec{x},n)$ accepts (rejects) $(q,Y)$ as well.
		\item For every $(p,X)$, $\vec{x}$, $n$, and $m$, if $(\vec{x},n)$ accepts (rejects) $(p|t,X)$ for all $t\in\ell(m,p)$, then $(\vec{x},n)$ accepts (rejects) $(p,X)$.
		\item For every $(p,X)$, $\vec{x}$, $n$, and $m$, there is a $q\leq_m p$ and a $Y\in\mathcal{H}\upharpoonright X$ such that for all $t\in\ell(m,q)$, $(\vec{x},n)$ decides $(q|t,Y)$.
	\end{enumerate}
\end{lemma}

\begin{proof}
	Parts (a) and (b) are immediate from the definitions above and the nature of the asymptotic game.
	
	(c) If $(\vec{x},n)$ accepts $(p|t,X)$ for all $t\in\ell(m,p)$, then by taking the pointwise maximum of each of the finitely many resulting strategies for I in $F[\vec{x},X]$ and using that $p=\bigcup_{t\in\ell(m,p)}p|t$, we have that $(\vec{x},n)$ accepts $(p,X)$. If $(\vec{x},n)$ rejects $(p|t,X)$ for all $t\in\ell(m,p)$, but there is some $(q,Y)\leq (p,X)$ which is accepted by $(\vec{x},n)$, then by Lemma \ref{lem:compat_restr}, there is a $t_0\in\ell(m,p)$ such that $q$ and $p|t_0$ are compatible. So, there is some $r$ below both $q$ and $p|t_0$ such that $(\vec{x},n)$ accepts $(r,Y)$, contradicting that $(\vec{x},n)$ rejects $(p|t_0,X)$.
	
	(d) Say $\ell(m,p)=\{t_0,\ldots,t_k\}$. Apply (a) to $(p|t_i,X)$ successively to obtain $(q_i,Y_i)\leq(p|t_i,X)$ which is decided by $(\vec{x},n)$, with $Y_i\preceq Y_{i-1}$, for $i=0,\ldots,k$, taking $Y_{-1}=X$. Let $q=\bigcup_{i\leq k} q_i\leq_m p$ and $Y=Y_k$. Then, $(q,Y)$ is as claimed.
\end{proof}

Fix an enumeration of $E^{[<\infty]}$ via some bijection $\#:E^{[<\infty]}\to\omega$.

\begin{lemma}\label{lem:decide}
	For every $(p,X)$, there is a $(q,Y)\leq(p,X)$ such that for all $(\vec{y},n)$, there is an $m$ such that for all $t\in\ell(m,q)$, $(\vec{y},n)$ decides $(q|t,Y)$.	
\end{lemma}

\begin{proof}
	By Lemma \ref{lem:accept_reject}(a), we may find $(q_0,Y_0)\leq(p,X)$ which is decided by $(\vec{y},0)$, where $\#(\vec{y})=0$. Suppose that $(q_i,Y_i)$ has already been defined for $i\leq m$. Let $\{(\vec{y}^0,\ell_0),\ldots,(\vec{y}^k,\ell_k)\}$ enumerate those finitely many $(\vec{y},\ell)$ with $\max\{ \#(\vec{y}),\ell\}\leq m+1$. Using Lemma \ref{lem:accept_reject}(d), we may successively choose $q^{j}\leq_{m+1} q^{j-1}\leq_{m+1} q_m$ and $Y^j\preceq Y^{j-1}\preceq Y_m$ in $\mathcal{H}$ such that for all $t\in\ell(m+1,q^j)$, $(\vec{y}^j,\ell_j)$ decides $(q^j|t,Y^j)$, for $j\leq k$. Let $q_{m+1}=q^k$ and $Y_{m+1}=Y^k$. Put $q=\bigcap_{m\in\omega} q_m$, the fusion of the $q_m$'s, and let $Y\in\mathcal{H}\upharpoonright X$ be a diagonalization of the $Y_m$'s. If $\max\{\#(\vec{y}),n\}\leq m$, then by construction, for all $t\in\ell(m,q)$, $(\vec{y},n)$ decides $(q|t,Y)$.
\end{proof}

\begin{lemma}\label{lem:accept_play_reject}
	For every $(p,X)$ and $\vec{x}$, there is a $(q,Y)\leq(p,X)$ such that either:
	\begin{enumerate}[(1)]
		\item I has a strategy $\sigma$ in $F[\vec{x},Y]$ such that $[q]\times[\sigma]\subseteq\mathbb{A}^c$, or
		\item for all $f\in [q]$, II has a strategy in $G[Y]$ for playing $(z_k)_{k\in\omega}$ for which there is an $n$ and a $t\sqsubset f$ such that $(\vec{x}^\smallfrown(z_0,\ldots,z_n),n)$ rejects $(q|t,Y)$.
	\end{enumerate}
\end{lemma}

\begin{proof}
	Let $(q',Y')\leq(p,X)$ be as in Lemma \ref{lem:decide}. We may assume that $Y'\preceq X/\vec{x}$. Let
	\begin{align*}
		\mathbb{B} =\{(f,(z_k)_{k\in\omega})\in[q']\times{E^{[\infty]}}&: \exists n\exists t\sqsubset f\\
		&[(\vec{x}^\smallfrown(z_0,\ldots,z_n),n) \text{ rejects } (q'|t,Y')]\}.
	\end{align*}
	Note that $\mathbb{B}$ is relatively open in $[q']\times{E^{[\infty]}}$. By Lemma \ref{lem:local_para_Rosendal_open}, there is a $(q,Y)\leq(q',Y')$ such that either I has a strategy $\sigma$ in $F[Y]$ such that $[q]\times[\sigma]\subseteq\mathbb{B}^c$, or for all $f\in[q]$, II has a strategy in $G[Y]$ for playing into $\mathbb{B}_f$. In the latter case, we're done, so we assume the former. 

	The strategy $\sigma_\emptyset=\sigma$ produces outcomes $(z_k)_{k\in\omega}$ in $F[Y]$ such that for all $n$, $(\vec{x}^\smallfrown(z_0,\ldots,z_n),n)$ does not reject, and thus accepts, $(q|t,Y)$ for all $t\in\ell(m,q)$, where $m$ is as in Lemma \ref{lem:decide} for $\vec{y}=\vec{x}^\smallfrown(z_0,\ldots,z_n)$. By Lemma \ref{lem:accept_reject}(c), $(\vec{x}^\smallfrown(z_0,\ldots,z_n),n)$ accepts $(q,Y)$. That is, I has a further strategy $\sigma_{(z_0,\ldots,z_n)}$ in $F[\vec{x}^\smallfrown(z_0,\ldots,z_n),Y]$ such that $[q]\times[\sigma_{(z_0,\ldots,z_n)}]\subseteq \mathbb{A}_n^c$. 
	
	We describe a strategy $\overline{\sigma}$ for I in $F[\vec{x},Y]$ such that $[q]\times[\overline{\sigma}]\subseteq\mathbb{A}^c$: if $(z_0,\ldots,z_{n-1})$ has been played by II thus far, I responds with
	\[
		\max\{\sigma_{\emptyset}(z_0,\ldots,z_{n-1}),\sigma_{(z_0)}(z_1,\ldots,z_{n-1}),\ldots,\sigma_{(z_0,\ldots,z_{n-1})}(\emptyset)\}.
	\]
	For each $n$, we have ensured that $[q]\times[\overline{\sigma}]\subseteq\mathbb{A}_n^c$, and thus $[q]\times[\overline{\sigma}]\subseteq\mathbb{A}^c$.
\end{proof}

Let $\{\mathbb{A}_s:s\in\omega^{<\omega}\}$ be a Suslin scheme of subsets of $2^\omega\times{E^{[\infty]}},$ where for each $s\in\omega^{<\omega}$, $\mathbb{A}_s=\bigcup_{n\in\omega}\mathbb{A}_{s^\smallfrown n}$. Given $\vec{x}$, $n$, and $s\in\omega^{<\omega}$, we define what it means for $(\vec{x},s,n)$ to \emph{accept} or \emph{reject} $(p,X)$ exactly as above, except with reference to $\mathbb{A}_s=\bigcup_{n\in\omega}\mathbb{A}_{s^\smallfrown n}$. Given $(q,Y)$, $\vec{y}$, and $s\in\omega^{<\omega}$, let
\begin{align*}
	\mathbb{B}(\vec{y},s,q,Y)=\{(f,(z_k)_{k\in\omega})\in [q]&\times{E^{[\infty]}}: \exists n\exists t\sqsubset f\\
	&[(\vec{y}^\smallfrown(z_0,\ldots,z_n),s,n)\text{ rejects }(q|t,Y)]\}.
\end{align*}
Our final lemma is a version of Lemma \ref{lem:accept_play_reject} which is uniform in $\vec{y}$ and $s$.

\begin{lemma}\label{lem:fusion_play_reject}
	For every $(p,X)$, there is a $(q,Y)\leq(p,X)$ such that for all $\vec{y}$, $s\in\omega^{<\omega}$, and $t\in q$, either:
	\begin{enumerate}[(1)]
		\item there exists an $r\leq q|t$ for which I has a strategy $\sigma$ in $F[\vec{y},Y]$ such that $[r]\times[\sigma]\subseteq\mathbb{A}_s^c$, or
		\item for all $f\in[q|t]$, II has a strategy in $G[Y]$ for playing into $\mathbb{B}(\vec{y},s,q,Y)_f$.
	\end{enumerate}
\end{lemma}

\begin{proof}
	For $\vec{y}=\emptyset$ and $s=\emptyset$, apply Lemma \ref{lem:accept_play_reject} to obtain a $(q_0,Y_0)\leq(p,X)$ such that either I has a strategy $\sigma$ in $F[Y_0]$ such that $[q_0]\times[\sigma]\subseteq\mathbb{A}_\emptyset^c$, or for all $f\in[q_0]$, II has a strategy in $G[Y_0]$ for playing into $\mathbb{B}(\emptyset,\emptyset,q_0,Y_0)_f$.
	
	Suppose we have defined $(q_i,Y_i)$ for $i\leq m$. Say $\ell(m+1,q_m)=\{t_0,\ldots,t_k\}$. Apply Lemma \ref{lem:accept_play_reject} to each $(q_m|t_j,Y_m)$ successively, for $j=0,\ldots, k$,  to obtain a $q^j\leq q_m|t_j$ and $Y^j\preceq Y^{j-1}\preceq Y_m$ in $\mathcal{H}$ such that for all of the finitely many $\vec{y}$ with $\#(\vec{y})\leq m+1$ and $s\in\omega^{<\omega}$ with $\max(s^\smallfrown|s|)\leq m+1$, either I has a strategy $\sigma$ in $F[\vec{y},Y^j]$ such that $[q^j]\times[\sigma]\subseteq\mathbb{A}_s^c$, or for all $f\in[q^j]$, II has a strategy in $G[Y^j]$ for playing into $\mathbb{B}(\vec{y},s,q^j,Y^j)_f$. Let $q_{m+1}=\bigcup_{j\leq k} q^j\leq_{m+1} q_m$ and $Y_{m+1}=Y^k$. 
	
	Put $q=\bigcap_{m\in\omega} q_m$ and let $Y\in\mathcal{H}\upharpoonright X$ be a diagonalization of the $Y_m$'s. 

	Suppose we are given $\vec{y}$, $s\in\omega^{<\omega}$, and $t\in q$. Taking $m$ greater than $\max\{\max(s^\smallfrown|s|),\#(\vec{y}),|t|\}$, the construction above yields two possibilities: either there is some $u\in\ell(m+1,q)$ which extends $t$ and for which I has a strategy $\sigma$ in $F[\vec{y},Y]$ such that $[q|u]\times[\sigma]\subseteq\mathbb{A}_s^c$, proving (1). Or, for all $u\in\ell(m+1,q)$ which extend $t$ and all $f\in[q|u]$, II has a strategy in $G[Y]$ for playing into $\mathbb{B}(\vec{y},s,q,Y)_f$, proving (2).
\end{proof}

We can now complete the proof of Theorem \ref{thm:local_para_Rosendal}.

\begin{thm}\label{thm:analytic}
	Analytic sets are perfectly $\mathcal{H}$-strategically Ramsey.	
\end{thm}

\begin{proof}
	Let $\mathbb{A}\subseteq2^\omega\times{E^{[\infty]}}$ be analytic, as witnessed by a continuous function $F:\omega^\omega\to2^\omega\times {E^{[\infty]}}$ such that $F[\omega^\omega]=\mathbb{A}$. For each $s\in\omega^{<\omega}$, let $\mathbb{A}_s=F[\mathcal{N}_s]$, where $\mathcal{N}_s=\{x\in\omega^\omega:s\sqsubseteq x\}$, and note that $\mathbb{A}_s=\bigcup_{n\in\omega}\mathbb{A}_{s^\smallfrown n}$. All references to acceptance or rejection and the sets $\mathbb{B}(\vec{y},s,q,Y)$ will be with respect to the Suslin scheme $\{\mathbb{A}_s:s\in\omega^{<\omega}\}$, as above. Let $(p,X)$ and $\vec{x}$ be given.
	
	By Lemma \ref{lem:fusion_play_reject}, there is a $(q,Y)\leq(p,X)$ such that for all $\vec{y}$, $s\in\omega^{<\omega}$, and $t\in q$, either:
	\begin{enumerate}[(1$'$)]
		\item there is an $r\leq q|t$ for which I has a strategy $\sigma$ in $F[\vec{y},Y]$ such that $[r]\times[\sigma]\subseteq\mathbb{A}_s^c$, or
		\item for all $f\in[q|t]$, II has a strategy in $G[Y]$ for playing into $\mathbb{B}(\vec{y},s,q,Y)_f$.
	\end{enumerate}
	
	If (1$'$) holds for $\vec{y}=\vec{x}$, $s=\emptyset$, and some $t\in q$, then since $\mathbb{A}_\emptyset=\mathbb{A}$, we're done. Thus, we suppose that (2$'$) holds for $\vec{y}=\vec{x}$, $s=\emptyset$, and every $t\in q$. 
	
	Fix $f\in[q]$. We will describe a strategy for II in $G[\vec{x},Y]$ for playing into $\mathbb{A}_f$. By (2$'$), II has a strategy in $G[Y]$ for playing into $\mathbb{B}(\vec{x},\emptyset,q,Y)_f$; they may follow this strategy until $(z_0,\ldots,z_{n_0})$ has been played such that $(\vec{x}^\smallfrown(z_0,\ldots,z_{n_0}),\emptyset,n_0)$ rejects $(q|t_0,Y)$, for some $t_0\sqsubset f$. By our choice of $(q,Y)$, II has a strategy in $G[Y]$ for playing into $\mathbb{B}(\vec{x}^\smallfrown(z_0,\ldots,z_{n_0}),(n_0),q|t_0,Y)_f$. Then, II follows this strategy until a further $(z_{n_0+1},\ldots,z_{n_0+n_1+1})$ has been played so that $(\vec{x}^\smallfrown(z_0,\ldots,z_{n_0},\ldots,z_{n_0+n_1+1}),(n_0),n_1)$ rejects $(q|t_1,Y)$, for some $t_1\sqsubset f$, which we may assume properly extends $t_0$, and so on.
	
	Let $m_k=(\sum_{j\leq k}n_j)+k$. The outcome of II's strategy just described will be a sequence $\vec{x}^\smallfrown Z$, where
	\[
		Z=(z_0,z_1,\ldots,z_{m_0},\ldots,z_{m_1},\ldots),
	\]
	such that for all $k\in\omega$, $(\vec{x}^\smallfrown(z_0,\ldots,z_{m_k}),(n_0,\ldots,n_{k-1}),n_k)$ rejects $(q|t_k,Y)$, $t_k\sqsubset f$.  In particular, there is some $Z^k\in E^{[\infty]}$, with $(z_0,\ldots,z_{m_k})\sqsubset Z^k$, and $f_k\in[q|t_k]$ such that $(f_k,\vec{x}^\smallfrown Z^k)\in\mathbb{A}_{(n_0,\ldots,n_k)}=F[\mathcal{N}_{(n_0,\ldots,n_k)}]$. Take $\beta_k\in\mathcal{N}_{(n_0,\ldots,n_k)}$ such that $F(\beta_k)=(f_k,\vec{x}^\smallfrown Z^k)$. Then, $\beta_k$ converges to $ (n_0,n_1,\ldots)$, $f_k$ to $f$, and $Z^k$ to $Z$ as $k\to\infty$, so by the continuity of $F$, $(f,\vec{x}^\smallfrown Z)=F(n_0,n_1,\ldots)\in\mathbb{A}$ and $\vec{x}^\smallfrown Z\in\mathbb{A}_f$.
\end{proof}

Finally, we note that, provided $|F|>2$, this conclusion for analytic sets is sharp in the sense that it is consistent relative to $\mathsf{ZFC}$ that there is a coanalytic subset $\mathbb{A}\subseteq E^{[\infty]}$ for which Theorem \ref{thm:Rosendal} fails (cf.~Section 6 of \cite{MR2179777} where this is proved for the analogous dichotomy in Banach spaces).

\section{Strategic filters and preservation}\label{sec:strat_pres}

Theorem \ref{thm:Sacks_pres} asserts that whenever $\mathcal{F}$ is a \emph{strategic} $(p^+)$-filter, it generates a \emph{strong} $(p^+)$-filter in the generic extension obtained after Sacks forcing. Whether the strategic property is itself preserved, or is necessary for this result, is the focus of this penultimate section.

When $|F|=2$, $(p^+)$-filters are necessarily strategic (Corollary 2.4 in \cite{MR4516184}), and so, by Theorem \ref{thm:Sacks_pres} above, are preserved by Sacks forcing. Alternatively, in this case nonzero vectors can be identified with their supports in $\mathrm{FIN}$, and $(p^+)$-filters are exactly stable ordered union ultrafilters (Theorem 6.3 in \cite{MR3864398}), which are equivalent to the \emph{selective} ultrafilters on $\mathrm{FIN}$ (by Corollary 2.5 of \cite{MR4516184}) shown to be preserved by Sacks forcing in \cite{MR3717938}.

In an earlier draft of this article, we claimed to show that the strategic property was \emph{not} preserved by Sacks forcing whenever $|F|>2$. However, an error in the proof was discovered by the anonymous referee, who then proposed Theorem \ref{thm:strong_strat} below. All of the results in the section are therefore due to the anonymous referee, to whom we again express our gratitude, and are presented with their kind permission.

We must first describe a variant of the Gowers game: given a filter $\mathcal{F}$ on $E$ and $X\in\mathcal{F}$, the \emph{restricted Gowers game} played below $X$, $G_{\mathcal{F}}[X]$, is defined in exactly the same way as the original Gowers game $G[X]$ except that I is restricted to playing block sequences $Y_k\in\mathcal{F}\upharpoonright X$. A \emph{strategy} $\alpha$ for player II in $G_{\mathcal{F}}[X]$ is defined as before, with $[\alpha]$ denoting the set of outcomes for this strategy. For $\vec{x}\in E^{[<\infty]}$, the game $G_{\mathcal{F}}[\vec{x},X]$ is defined analogously to $G[\vec{x},X]$. The following lemma allows us to pass from $G_{\mathcal{F}}[X]$ to $G[Y]$ for some $Y\in\mathcal{F}\upharpoonright X$, whenever $\mathcal{F}$ is a $(p^+)$-filter:

\begin{lemma}\label{lem:unrestricted}
	Let $\mathcal{F}$ be a $(p^+)$-filter on $E$, $X\in\mathcal{F}$, and $\alpha$ a strategy for II in $G_{\mathcal{F}}[X]$. Then, there exists a $Y\in\mathcal{F}\upharpoonright X$ and a strategy $\beta$ for II in $G[Y]$ such that $[\beta]\subseteq[\alpha]$.
\end{lemma}

\begin{proof}
	By Lemma 5.4 in \cite{MR4516184}, there is a \emph{tree} $T\subseteq E^{[<\infty]}$, meaning a set of finite block sequences closed under initial segments, satisfying:
	\begin{enumerate}[(i)]
		\item $[T]\subseteq[\alpha]$, where $[T]$ is the set of infinite branches through $T$, and
		\item whenever $\vec{x}\in T$ and $Y\in\mathcal{F}$, there is a $y\in\langle Y\rangle$ such that $\vec{x}^\smallfrown y\in T$.
	\end{enumerate}
	Given $\vec{x}\in T$, let 
	\[
		T_{\vec{x}}=\{y\in E^\times:\vec{x}^\smallfrown y\in T\}.
	\]
	Condition (ii) above says that for no $Y\in\mathcal{F}$ is it the case that $T_{\vec{x}}\cap\langle Y\rangle=\emptyset$. Since $\mathcal{F}$ is full, there must exist a $Y_{\vec{x}}\in\mathcal{F}$ such that $T_{\vec{x}}$ is asymptotic below $Y_{\vec{x}}$, by Lemma \ref{lem:full_pigeon}. Let $Y\in\mathcal{F}\upharpoonright X$ be a diagonalization of the $Y_{\vec{x}}$'s as $\vec{x}$ ranges over $T$ (to do this, enumerate the $\vec{x}$'s as $\vec{x}_n$'s and take successive block sequences in $\mathcal{F}$ which are contained in $\langle Y_{\vec{x}_0}\rangle\cap\cdots\cap \langle Y_{\vec{x}_n}\rangle$, which is possible since $\mathcal{F}$ is a filter, and then diagonalize). So, for every $\vec{x}\in T$, $T_{\vec{x}}$ is asymptotic below $Y$. Thus, we can define a strategy $\beta$ for II in $G[Y]$ as follows: assuming I and II have played $Y_0,\ldots,Y_{n-1}\preceq Y$ and $x_0,\ldots,x_{n-1}$, respectively, thus far, if I now plays $Y_n\preceq Y$, since $T_{(x_0,\ldots,x_{n-1})}$ is asymptotic below $Y$, II may respond with some $x_n\in T_{(x_0,\ldots,x_{n-1})}\cap\langle Y_n\rangle$. By construction, we have that $[\beta]\subseteq[T]$, and thus $[\beta]\subseteq[\alpha]$.
\end{proof}

The restricted Gowers games enable a form of Theorem \ref{thm:local_Rosendal} for $(p)$-families $\mathcal{H}$ without the additional assumption of fullness: whenever $\mathcal{H}$ is a $(p)$-family on $E$ and $\mathbb{A}\subseteq E^{[\infty]}$ is analytic, then for every $\vec{x}\in E^{[<\infty]}$ and $X\in\mathcal{H}$, there is a $Y\in\mathcal{H}\upharpoonright X$ such that either:
\begin{enumerate}[(1)]
	\item I has a strategy in $F[\vec{x},Y]$ for playing out of $\mathbb{A}$, or
	\item II has a strategy in $G_{\mathcal{H}}[\vec{x},Y]$ for playing into $\mathbb{A}$.
\end{enumerate}
In other words, $\mathbb{A}$ is $\mathcal{H}$-strategically Ramsey, but with the strategy in (2) weakened to being one for the restricted Gowers game; see Theorem 3.11.5 in \cite{Smythe_thesis} or, in more generality, Theorem 3.3 in \cite{MR4127870}.

Under the additional assumption of the existence of a measurable cardinal, it is further shown in Theorem II.36 of \cite{deRancourt_thesis}\footnote{While this result has not been published outside of the thesis \cite{deRancourt_thesis}, it can be proved using methods similar to that of Theorem 3.3 in \cite{MR4127870}.} that the same conclusion holds for all $\mathbf{\Sigma}^1_2$ sets $\mathbb{A}\subseteq E^{[\infty]}$. If $\mathcal{H}$ is, moreover, a full filter, we can then apply Lemma \ref{lem:unrestricted} to obtain the following:

\begin{thm}\label{thm:Sigma12}
	Assume the existence of a measurable cardinal. Let $\mathcal{F}$ be a $(p^+)$-filter on $E$. Then, $\mathbf{\Sigma}^1_2$ sets are $\mathcal{F}$-strategically Ramsey.\qed	
\end{thm}

As in the proof of Lemma \ref{lem:local_para_Rosendal_open}, we can force to add a $(p^+)$-filter inside any $(p^+)$-family $\mathcal{H}$, without adding reals, and it will follow that $\mathbf{\Sigma}^1_2$ sets are $\mathcal{H}$-strategically Ramsey, under the same large cardinal assumption.

We can now apply Theorem \ref{thm:Sigma12} to show that every strong $(p^+)$-filter is, in fact, strategic, in the presence of a measurable cardinal.

\begin{thm}\label{thm:strong_strat}
	Assume the existence of a measurable cardinal. Every strong $(p^+)$-filter on $E$ is strategic.	
\end{thm}

\begin{proof}
	Let $\mathcal{F}$ be a strong $(p^+)$-filter on $E$. Fix a strategy $\alpha$ for II in the finitized Gowers game $G^{<\infty}[X]$, for some $X\in\mathcal{F}$, and let
	\[
		\mathbb{D}=\{Z\in E^{[\infty]}:\exists Y\in[\alpha](Z\preceq Y)\}.
	\]
	By Lemma 4.7 of \cite{MR3864398}, $\mathbb{D}$ is dense open (in the sense of forcing with respect to $\preceq$) below $X$, and in order to show that $\mathcal{F}$ is strategic, it suffices to prove that $\mathbb{D}\cap\mathcal{F}\neq\emptyset$.
	
	Since $\mathbb{D}$ is dense open below $X$, for every $Y\preceq X$, I has a strategy in $G[Y]$ for playing into $\mathbb{D}$; simply play a fixed element of $\mathbb{D}$ below $Y$ repeatedly. As $\mathbb{D}^c$ is coanalytic, and in particular $\mathbf{\Sigma}^1_2$, it is $\mathcal{F}$-strategically Ramsey by Theorem \ref{thm:Sigma12}. Thus, there exists a $Y\in\mathcal{F}\upharpoonright X$ such that I has a strategy $\sigma$ in $F[Y]$ for playing into $\mathbb{D}$. Finally, by Theorem 4.3 in \cite{MR3864398} and the fact that $\mathcal{F}$ is a strong $(p^+)$-filter, $[\sigma]\cap\mathcal{F}\neq\emptyset$, and thus $\mathbb{D}\cap\mathcal{F}\neq\emptyset$. Hence, $\mathcal{F}$ is strategic.
\end{proof}

Consequently, in the presence of sufficient large cardinals, strong $(p^+)$-filters have complete combinatorics, in the sense that they are generic over $\mathbf{L}(\mathbb{R})$ for $(E^{[\infty]},\preceq)$, by Theorem 1.2 of \cite{MR3864398}.

We now come to the preservation result, which can be viewed as a sharp form of Theorem \ref{thm:Sacks_pres} in the presence of a measurable cardinal.

\begin{cor}\label{cor:strat_pres}
	Assume the existence of a measurable cardinal. If $\mathcal{F}$ is a strategic $(p^+)$-filter on $E$ and $g$ is $\mathbf{V}$-generic for $\mathbb{S}$, then $\mathcal{F}$ generates a strategic $(p^+)$-filter in $\mathbf{V}[g]$.	
\end{cor}

\begin{proof}
	By Theorem \ref{thm:Sacks_pres}, $\mathcal{F}$ generates a strong $(p^+)$-filter $\overline{\mathcal{F}}$ in $\mathbf{V}[g]$. By the L\'{e}vy--Solovay Theorem (Theorem 21.1 in \cite{MR1940513}), the measurable cardinal remains measurable in $\mathbf{V}[g]$, and so $\overline{\mathcal{F}}$ is strategic by Theorem \ref{thm:strong_strat}.
\end{proof}

Finally, we note that the same conclusion holds after adding finitely many Sacks reals, iteratively.

\section{Coda}\label{sec:end}

We end here with a discussion of the ways Theorems \ref{thm:local_para_Rosendal} and \ref{thm:Sacks_pres} could be improved and some ideas for how to do so. First and foremost, there is the hypothesis of being ``strategic'' in Theorems \ref{thm:local_para_Rosendal} and \ref{thm:Sacks_pres}. This assumption is not present in Theorem \ref{thm:local_Rosendal}, the author's original local Ramsey theorem for block sequences in $E$, and we do not know whether it is preserved by Sacks forcing, without additional large cardinal assumptions.

\begin{ques}\label{ques:strat}
	Is ``strategic'' necessary in Theorems \ref{thm:local_para_Rosendal} and \ref{thm:Sacks_pres}? Can it be reduced to ``strong'', or even just ``$(p^+)$'', in $\mathsf{ZFC}$?
\end{ques}

The strategic property is used in only one part of the proofs of Theorems \ref{thm:local_para_Rosendal} and \ref{thm:Sacks_pres}, namely in the proof of Theorem \ref{thm:local_para_1d}, the local version of the parametrized weak pigeonhole principle ($\dagger$), via an application of Lemma \ref{lem:H_dense}. Note that fullness (Proposition 3.6 in \cite{MR3864398}) and the $(p)$-property (Theorem 4.1 in \cite{MR4516184}) are both necessary for the clopen case of Theorem \ref{thm:local_Rosendal}, and so cannot be removed from the hypotheses of Theorem \ref{thm:local_para_Rosendal} either.

We believe that Question \ref{ques:strat} would be resolved by finding the ``right'' proof of Theorem \ref{thm:local_para_Rosendal}. What we have in mind is a combinatorial (rather than metamathematical) argument in the style of Ellentuck's proof \cite{MR0349393} of the Galvin--Prikry and Silver Theorems, and Pawlikowki's proof \cite{MR1075374} of the Miller--Todor\v{c}evi\'{c} Theorem, combined with arguments from \cite{MR2604856} and \cite{MR3864398} specific to the vector space setting.

\begin{ques}\label{ques:Ellentuck}
	Is there an ``Ellentuck-style'' proof of Theorem \ref{thm:local_para_Rosendal}?
\end{ques}

The arguments in Section \ref{sec:analytic} are the remnants of such a proof and show that one can go from open sets being perfectly $\mathcal{H}$-strategically Ramsey to analytic sets using only the $(p)$-property of $\mathcal{H}$ and fusion in $\mathbb{S}$. Thus, it would suffice to give an ``Ellentuck-style'' proof of Lemma \ref{lem:local_para_Rosendal_open}.

Baumgartner and Laver \cite{MR556894} showed that selective ultrafilters are preserved by iterated Sacks forcing. That is, after an $\omega_2$-length countable support iteration of Sacks forcing, any selective ultrafilters in the ground model still generate selective ultrafilters in the extension. The presence of ``strategic'' in the hypotheses of Theorem \ref{thm:Sacks_pres}, and its absence in the conclusion, may present an obstacle to proving an analogous result for iterated Sacks forcing here in $\mathsf{ZFC}$ (it may be easier to show such a result under the existence of a measurable cardinal, using Corollary \ref{cor:strat_pres}).

\begin{ques}\label{ques:iteration}
	If $\mathcal{F}$ is a strategic $(p^+)$-filter on $E$, $\alpha\leq\omega_2$, and $G_\alpha$ is $\mathbf{V}$-generic for an $\alpha$-length countable support iteration $\mathbb{S}_{\alpha}$ of Sacks forcing, must $\mathcal{F}$ generate a strong $(p^+)$-filter in $\mathbf{V}[G_\alpha]$?
\end{ques}

Next is the question of improving the conclusion of Theorem \ref{thm:local_para_Rosendal}. The most natural extension would be to parametrize by a countable sequence of perfect sets, rather than just a single perfect set.

\begin{ques}\label{ques:prod_pres}
	Can the Ramsey theory of block sequences in $E$ be parametrized by countable sequences of perfect sets? That is, given an analytic set $\mathbb{A}\subseteq \mathbb{R}^\omega\times E^{[\infty]}$, must there exist a sequence $(P_n)_{n\in\omega}$ of nonempty perfect sets in $\mathbb{R}$ and an $X\in E^{[\infty]}$ such that either:
	\begin{enumerate}[(1)]
		\item I has a strategy $\sigma$ in $F[X]$ such that $(\prod_{n\in\omega} P_n)\times[\sigma]\subseteq\mathbb{A}^c$, or
		\item for every $\overline{t}=(t_n)_{n\in\omega}\in \prod_{n\in\omega} P_n$, II has a strategy in $G[X]$ for playing into $\mathbb{A}_{\overline{t}}$?
	\end{enumerate}
	If so, can this be localized to a (strategic) $(p^+)$-family on $E$?
\end{ques}

Theorem \ref{thm:para_Silver} can be extended to countable sequences of perfect sets using Laver's infinite-dimensional form \cite{MR754925} of the Halpern--L\"{a}uchli Theorem \cite{MR200172}. Laver used this to show that selective ultrafilters are preserved by forcing with arbitrarily long countable support products, or ``side-by-side'', Sacks forcing. The corresponding facts for finite products of perfects sets and Sacks forcing were proven earlier by Pincus and Halpern \cite{MR626489}. Zheng \cite{MR3717938} and Kawach \cite{MR4247793} have shown that $\mathrm{FIN}$, $\mathrm{FIN}_k$, and $\mathrm{FIN}_{\pm k}$ can each be parametrized by countable sequences of perfect sets, and Zheng also established the corresponding ultrafilter preservation result for $\mathrm{FIN}$. We ask if such a preservation result is possible in our setting:

\begin{ques}\label{ques:prod_Sacks_pres}
	If $\mathcal{F}$ is a strategic $(p^+)$-filter on $E$, $\kappa$ a cardinal, and $G$ is $\mathbf{V}$-generic for a $\kappa$-length countable support product $\mathbb{S}^{(\kappa)}$ of Sacks forcing, must $\mathcal{F}$ generate a strong $(p^+)$-filter in $\mathbf{V}[G]$?
\end{ques}

A ``Yes'' answer to the local form of Question \ref{ques:prod_pres} would yield a ``Yes'' answer to Question \ref{ques:prod_Sacks_pres}, though as we have seen in the proofs of Theorems \ref{thm:local_para_Rosendal} and \ref{thm:Sacks_pres}, their proofs could be intertwined in some way.

Again, we expect that the key to resolving Questions \ref{ques:prod_pres} and \ref{ques:prod_Sacks_pres} lies in Question \ref{ques:Ellentuck}. If we were able to find an ``Ellentuck-style'' proof of Theorem \ref{thm:local_para_Rosendal}, then by adapting any fusion arguments in $\mathbb{S}$ to the countable product $\mathbb{S}^{(\omega)}$, as in \cite{MR794485}, we should be able to extend the proof to answer Question \ref{ques:prod_pres} in the affirmative. Finally, following Zheng's proof that stable ordered union ultrafilters in $\mathrm{FIN}$ are preserved by iterated Sacks forcing (Theorem 5.42 in \cite{yyz_thesis}), we expect that a ``Yes'' answer to Question \ref{ques:prod_pres} would assist in resolving Question \ref{ques:iteration} as well.



\bibliographystyle{abbrv}
\bibliography{../math_bib}{}

\begin{thebibliography}{10}

\bibitem{MR1839387}
J.~Bagaria and J.~L{{\'o}}pez-Abad.
\newblock Weakly {R}amsey sets in {B}anach spaces.
\newblock {\em Adv. Math.}, 160(2):133--174, 2001.

\bibitem{MR794485}
J.~E. Baumgartner.
\newblock Sacks forcing and the total failure of {M}artin's axiom.
\newblock {\em Topology Appl.}, 19(3):211--225, 1985.

\bibitem{MR556894}
J.~E. Baumgartner and R.~Laver.
\newblock Iterated perfect-set forcing.
\newblock {\em Ann. Math. Logic}, 17(3):271--288, 1979.

\bibitem{MR4391477}
D.~Calder\'{o}n, C.~A. Di~Prisco, and J.~G. Mijares.
\newblock Ramsey subsets of the space of infinite block sequences of vectors.
\newblock {\em Fund. Math.}, 257(2):189--216, 2022.

\bibitem{deRancourt_thesis}
N.~de~Rancourt.
\newblock {\em Th\'{e}orie de Ramsey sans principe des tiroirs et applications
  \`{a} la preuve de dichotomies d'espaces de Banach}.
\newblock PhD thesis, Universit\'{e} Paris-Diderot, 2018.

\bibitem{MR4127870}
N.~de~Rancourt.
\newblock Ramsey theory without pigeonhole principle and the adversarial
  {R}amsey principle.
\newblock {\em Trans. Amer. Math. Soc.}, 373(7):5025--5056, 2020.

\bibitem{MR0349393}
E.~Ellentuck.
\newblock A new proof that analytic sets are {R}amsey.
\newblock {\em J. Symbolic Logic}, 39:163--165, 1974.

\bibitem{MR1644345}
I.~Farah.
\newblock Semiselective coideals.
\newblock {\em Mathematika}, 45(1):79--103, 1998.

\bibitem{MR0337630}
F.~Galvin and K.~Prikry.
\newblock Borel sets and {R}amsey's theorem.
\newblock {\em J. Symbolic Logic}, 38:193--198, 1973.

\bibitem{MR2155534}
S.~Geschke and S.~Quickert.
\newblock On {S}acks forcing and the {S}acks property.
\newblock In {\em Classical and new paradigms of computation and their
  complexity hierarchies}, volume~23 of {\em Trends Log. Stud. Log. Libr.},
  pages 95--139. Kluwer Acad. Publ., Dordrecht, 2004.

\bibitem{MR1164759}
W.~T. Gowers.
\newblock Lipschitz functions on classical spaces.
\newblock {\em European J. Combin.}, 13(3):141--151, 1992.

\bibitem{MR1954235}
W.~T. Gowers.
\newblock An infinite {R}amsey theorem and some {B}anach-space dichotomies.
\newblock {\em Ann. of Math. (2)}, 156(3):797--833, 2002.

\bibitem{MR0306010}
R.~L. Graham, K.~Leeb, and B.~L. Rothschild.
\newblock Ramsey's theorem for a class of categories.
\newblock {\em Adv. Math.}, 8:417--433, 1972.

\bibitem{MR200172}
J.~D. Halpern and H.~L\"{a}uchli.
\newblock A partition theorem.
\newblock {\em Trans. Amer. Math. Soc.}, 124:360--367, 1966.

\bibitem{MR0349574}
N.~Hindman.
\newblock Finite sums from sequences within cells of a partition of
  {$\mathbb{N}$}.
\newblock {\em J. Combinatorial Theory Ser. A}, 17:1--11, 1974.

\bibitem{MR1940513}
T.~Jech.
\newblock {\em Set theory}.
\newblock Springer Monographs in Mathematics. Springer-Verlag, Berlin, 2003.

\bibitem{MR4247793}
J.~K. Kawach.
\newblock Parametrized {R}amsey theory of infinite block sequences of vectors.
\newblock {\em Ann. Pure Appl. Logic}, 172(8):102984, 22, 2021.

\bibitem{MR1321597}
A.~S. Kechris.
\newblock {\em Classical descriptive set theory}, volume 156 of {\em Graduate
  Texts in Mathematics}.
\newblock Springer-Verlag, New York, 1995.

\bibitem{MR597342}
K.~Kunen.
\newblock {\em Set theory}, volume 102 of {\em Studies in Logic and the
  Foundations of Mathematics}.
\newblock North-Holland, Amsterdam, 1980.

\bibitem{MR2737185}
C.~Laflamme, L.~Nguyen Van~Th\'e, M.~Pouzet, and N.~Sauer.
\newblock Partitions and indivisibility properties of countable dimensional
  vector spaces.
\newblock {\em J. Combin. Theory Ser. A}, 118(1):67--77, 2011.

\bibitem{MR754925}
R.~Laver.
\newblock Products of infinitely many perfect trees.
\newblock {\em J. London Math. Soc. (2)}, 29(3):385--396, 1984.

\bibitem{MR2179777}
J.~L{{\'o}}pez-Abad.
\newblock Coding into {R}amsey sets.
\newblock {\em Math. Ann.}, 332(4):775--794, 2005.

\bibitem{MR0491197}
A.~R.~D. Mathias.
\newblock Happy families.
\newblock {\em Ann. Math. Logic}, 12(1):59--111, 1977.

\bibitem{MR2285192}
J.~G. Mijares.
\newblock Parametrizing the abstract {E}llentuck theorem.
\newblock {\em Discrete Math.}, 307(2):216--225, 2007.

\bibitem{MR2875907}
J.~G. Mijares and J.~E. Nieto.
\newblock A parametrization of the abstract {R}amsey theorem.
\newblock {\em Divulg. Mat.}, 16(2):259--274, 2008.

\bibitem{MR983001}
A.~W. Miller.
\newblock Infinite combinatorics and definability.
\newblock {\em Ann. Pure Appl. Logic}, 41(2):179--203, 1989.

\bibitem{MR0373906}
K.~R. Milliken.
\newblock Ramsey's theorem with sums or unions.
\newblock {\em J. Combinatorial Theory Ser. A}, 18:276--290, 1975.

\bibitem{MR1075374}
J.~Pawlikowski.
\newblock Parametrized {E}llentuck theorem.
\newblock {\em Topology Appl.}, 37(1):65--73, 1990.

\bibitem{MR626489}
D.~Pincus and J.~D. Halpern.
\newblock Partitions of products.
\newblock {\em Trans. Amer. Math. Soc.}, 267(2):549--568, 1981.

\bibitem{MR2604856}
C.~Rosendal.
\newblock An exact {R}amsey principle for block sequences.
\newblock {\em Collect. Math.}, 61(1):25--36, 2010.

\bibitem{MR0332480}
J.~Silver.
\newblock Every analytic set is {R}amsey.
\newblock {\em J. Symbolic Logic}, 35:60--64, 1970.

\bibitem{Smythe:ParaRamseyBlockII}
I.~B. Smythe.
\newblock Parametrizing the {R}amsey theory of vector spaces {II}: {B}anach
  spaces.
\newblock In preperation.

\bibitem{Smythe_thesis}
I.~B. Smythe.
\newblock {\em Set theory in infinite-dimensional vector spaces}.
\newblock PhD thesis, Cornell University, 2017.

\bibitem{MR3864398}
I.~B. Smythe.
\newblock A local {R}amsey theory for block sequences.
\newblock {\em Trans. Amer. Math. Soc.}, 370(12):8859--8893, 2018.

\bibitem{MR4045990}
I.~B. Smythe.
\newblock Madness in vector spaces.
\newblock {\em J. Symb. Log.}, 84(4):1590--1611, 2019.

\bibitem{MR4516184}
I.~B. Smythe.
\newblock Filters on a countable vector space.
\newblock {\em Fund. Math.}, 260(1):41--58, 2023.

\bibitem{MR2603812}
S.~Todor{\v{c}}evi{{\'c}}.
\newblock {\em Introduction to {R}amsey spaces}, volume 174 of {\em Annals of
  Mathematics Studies}.
\newblock Princeton University Press, Princeton, NJ, 2010.

\bibitem{MR891249}
F.~van Engelen, A.~W. Miller, and J.~Steel.
\newblock Rigid {B}orel sets and better quasi-order theory.
\newblock In {\em Logic and combinatorics ({A}rcata, {C}alif., 1985)},
  volume~65 of {\em Contemp. Math.}, pages 199--222. Amer. Math. Soc.,
  Providence, RI, 1987.

\bibitem{MR3717938}
Y.~Y. Zheng.
\newblock Selective ultrafilters on {FIN}.
\newblock {\em Proc. Amer. Math. Soc.}, 145(12):5071--5086, 2017.

\bibitem{yyz_thesis}
Y.~Y. Zheng.
\newblock {\em Parametrizing topological {R}amsey spaces}.
\newblock PhD thesis, University of Toronto, 2018.

\end{thebibliography}





\end{document}